\documentclass{amsart}

\usepackage{ebproof, multicol,enumerate, amssymb, amsthm, txfonts, xytree, mathdots}

\usepackage{stmaryrd}
\usepackage{xy}
\usepackage{framed}

\newtheorem{theorem}{Theorem}
\theoremstyle{definition}
	\newtheorem{lemma}[theorem]{Lemma}
	\newtheorem{cor}[theorem]{Corollary}
	\newtheorem{mdef}[theorem]{Definition}

     \newcommand{\LL}{\mathcal{L}}
     \newcommand{\At}{\mathsf{At}}
     \newcommand{\B}{\mathbf{B}}
     \newcommand{\R}{\mathbf{R}}
     
     \newcommand{\bun}{\mathcal{B}}
     \newcommand{\seq}{\mathsf{seq}}
     \newcommand{\rseq}{\mathsf{rseq}}
     \newcommand{\nd}[1]{\mathbf{#1d}}
     \newcommand{\ndd}[1]{\mathbf{#1dr}}
     \newcommand{\squigto}{\rightsquigarrow}
     \newcommand{\Bpd}{\mathbf{Bpd}}
     \newcommand{\cf}{\mathsf{cf}}
    \newcommand{\red}{\mathsf{red}}
	
\title{Hyperformalism for Bunched Natural Deduction Systems}

\author{Shay Allen Logan}
\email{salogan@ksu.edu}

\author{Blaine Worley}
\email{bdworley@ucdavis.edu}

\begin{document}
    
\begin{abstract}
    Logics closed under classes of substitutions broader than class of uniform substitutions are known as hyperformal logics. This paper extends known results about hyperformal logics in two ways. First: we examine a very powerful form of hyperformalism that tracks, for bunched natural deduction systems, essentially all the intensional content that can possibly be tracked. We demonstrate that, after a few tweaks, the well-known relevant logic $\mathbf{B}$ exhibits this form of hyperformalism. Second: we demonstrate that not only can hyperformalism be extended along these lines, it can also be extended to accommodate not just what is proved in a given logic but the proofs themselves. Altogether, the paper demonstrates that the space of possibilities for the study of hyperformalism is much larger than might have been expected.
\end{abstract}\maketitle

    Formal logics are closed under uniform substitutions. What about logics that---in addition to being closed under uniform substitutions---are closed under nonuniform substitutions as well? In \cite{logan2022} it was shown that some naturally-arising relevant logics had such a feature. In virtue of the fact that this intuitively makes such logics \textit{yet more formal}, they were labeled `hyperformal'. 

    Hyperformalism looks at first glance like an undesirable feature for a logic to have. We require that logics be \textit{formal} because we recognize that logic alone cannot guarantee that distinct atoms have the same meaning, and thus cannot prevent our replacing them with distinct formulas. The hyperformal logic investigated in \cite{logan2022} allows us to go further and replace the same atom occurring in appropriately different contexts with distinct formulas. But it seems logic alone \textit{should} be enough to guarantee that the distinct occurrences of the same atom---no matter the context---are nonetheless distinct occurrences of things that have the same meaning. 

    This can be challenged along roughly the following lines. First: topic, as argued in e.g. \cite{Yablo2014} or \cite{Berto2022} seems to be intimately connected to meaning. Shifts in topic might thus plausibly be taken to induce shifts in meaning. Second: there's a long tradition---a tradition made most explicit in works like \cite{Angell1977}, \cite{Angell1989}, and \cite{Fine2016}---of supposing that certain intensional operators like the negation and the conditional are \textit{topic transformative}. On the thesis that topics and meanings are intimately connected, then such operators might plausibly be taken to be meaning-transformative. If so, atoms occurring in distinct intensional contexts cannot be assumed to have the same meaning. So logic alone cannot guarantee that they in fact do have the same meaning and must thus permit that they be replaced by distinct formulas.

    That was, at best, a sketchy reformulation of a few strands of serious philosophical work. Whether it hangs together isn't entirely clear. The reader who finds it sketchy ought to think of this as an invitation to flesh it out themselves.

    Regardless of whether the story works, the following is true: given different classes of not-necessarily-uniform substitutions, we obtain different flavors of hyperformalism. The first alternative to the `depth hyperformalism' of \cite{logan2022} was presented in \cite{ferguson2023topic}. Based on what the authors know of investigations ongoing in the community, it seems there are a good number of further such flavors forthcoming. This paper, in part, makes a contribution to this proliferation of flavors of hyperformalism.
    
    But it also extends the discussion along another dimension. Worley showed in~\cite{worley2024} that depth substitutions act on Hilbert proofs in a natural way, and that in fact sets of Hilbert proofs in certain relevant logics are also depth hyperformal. This result is tantalizing, as it suggests that hyperformalism might be both a more natural and a more omnipresent phenomenon than it initially appeared to be. Tantalizing as Worley's result is, the fact that it was limited to consideration of Hilbert proofs---a thoroughly suboptimal sort of proof-theoretic object---was a major limitation. So, in part to make good on the tantalizingness of Worley's first result, we'll first turn to showing that a number of classes of more natural proof-theoretic objects associated with relevance logics are also depth hyperformal. 

    Before turning to a summary of the paper, we return briefly to the topic of topic. Readers who found themselves skeptical of the topic motivation should note the following about the technical work in the paper. First, the logics we will be investigating---and which we will show possess striking forms of hyperformalism---are not logics tailor-made for this job. Indeed, in every case, they are logics that have arisen quite naturally in other contexts. Fans of these logics who reject the topical motivation given above are thus left with an explanatory burden: given that these logics are in fact hyperformal in the identified ways, if it's not topical considerations that explain it, then what is it that does?

    The outline of the paper is as follows. In Section 1 we set things up by defining our language and the several logics we will be working with. We recap known results about variable sharing and hyperformalism in Section 2. Section 3 extends these results to the proofs in the bunched natural deduction systems presented in \cite{Read1988}. We then turn to a second axis of extension in Section 4, where we present a much broader class of substitutions these systems are closed under. We discuss the results in the conclusion. A pair of appendices contain proofs of results that, while important in the paper, are tangential to its main aims.

\section{Setup}

    We define the sets $\LL$ and $\bun$ as follows:
    \begin{itemize}
        \item $\At=\{p_i\}_{i=1}^\infty\subseteq\LL$. The members of $\At$ are called \textit{atomic} formulas.
        \item If $ A \in\LL$ and $ B \in\LL$, then $\neg A \in\LL$, $( A \land B )\in\LL$, $( A \lor B )\in\LL$, $( A \to B )\in\LL$, and $( A \circ B )\in\LL$. Members of $\LL$ are called \textit{formulas}.
        \item $\LL\subseteq\bun$. Thought of as members of $\bun$, we call the members of $\LL$ \textit{atomic} bunches.
        \item If $ X \in\bun$ and $ Y \in\bun$, then $( X , Y )\in\bun$ and $( X ; Y )\in\bun$. Members of $\bun$ are called \textit{bunches}.
    \end{itemize}
    We will write $\LL^-$ for the $\circ$-free fragment of $\LL$, and we drop outermost parentheses as per the usual conventions. A \textit{consecution} has the form $\Gamma\Yright A$ with $\Gamma\in\bun$ and $A\in\LL$. 

    We will work with a few different logics throughout the paper. For convenience, we will introduce all of them here---some more than once! 

    The weakest logic we will consider is the logic $\mathbf{B}$ that was identified in \cite{RLR} as the \textbf{B}asic relevant logic and which is the main object investigated in \cite{logan2024}. Axiomatically, it can be presented as follows:
    \begin{enumerate}[{A}1)]
        \item $A\to A$
        \item $(A\land B)\to A$
        \item $(A\land B)\to B$
        \item $((A\to B)\land(A\to C))\to(A\to(B\land C))$
        \item $A\to(A\lor B)$ 
        \item $B\to(A\lor B)$
        \item $((A\to C)\land(B\to C))\to((A\lor B)\to C)$
        \item $(A\land(B\lor C))\to((A\land B)\lor(A\land C))$
        \item $\neg\neg A\to A$
    \end{enumerate}
    \begin{enumerate}[{R}1)]
        \item If $A$ is a theorem and $B$ is a theorem, then $A\land B$ is a theorem.\footnote{We're saying these `out loud' as it were because the comma and the turnstile have already become overloaded and we don't want to cause any more confusion than we have to.}
        \item If $A\to B$ is a theorem and $A$ is a theorem, then $B$ is a theorem.
        \item If $A\to\neg B$ is a theorem, then $B\to\neg A$ is a theorem.
        \item If $A\to B$ and $C\to D$ are theorems, then $(B\to C)\to(A\to D)$ is a theorem. 
    \end{enumerate}
    We write $\mathbf{hB}$ for the set of theorems generated by these axioms and rules. As a bunch-formula natural deduction system, we can instead present this logic as in Figure~\ref{fig:cons}. Presented in this way, we will call the logic $\mathbf{B}$, rather than $\mathbf{hB}$. That $\mathbf{hB}$ and $\mathbf{B}$ are sufficiently similar to be regarded as the same logic is not obvious. Nonetheless, this has been proved in various places, including most prominently in \cite{Read1988}.

        \begin{figure}[htb]
    \begin{framed}
        \centering
        \textbf{Axiom}\\
        \vspace{.2in}
        \begin{prooftree}
            \infer0[(id)]{A \Yright A}
        \end{prooftree}\\
        \vspace{.2in}
        \textbf{Operational Rules}\\
        \vspace{.2in}
        \begin{prooftree}
            \hypo{X;A\Yright B}
            \infer1[($\to$I)]{X\Yright A\to B}
        \end{prooftree}
        \qquad
        \begin{prooftree}
            \hypo{X\Yright A\to B}
            \hypo{Y\Yright A}
            \infer2[($\to$E)]{X;Y\Yright B}
        \end{prooftree}\\
        \vspace{.2in}
        \begin{prooftree}
            \hypo{X\Yright A}
            \infer1[($\lor$I$_1$)]{X\Yright A\lor B}
        \end{prooftree}
        \qquad
        \begin{prooftree}
            \hypo{X\Yright B}
            \infer1[($\lor$I$_2$)]{X\Yright A\lor B}
        \end{prooftree}\\
        \vspace{.2in}
        \begin{prooftree}
            \hypo{X\Yright A\lor B}
            \hypo{Y(A)\Yright C}
            \hypo{Y(B)\Yright C}
            \infer3[($\lor$E)]{Y(X)\Yright C}
        \end{prooftree}\\
        \vspace{.2in}
        \begin{prooftree}
            \hypo{X\Yright A}
            \hypo{Y\Yright B}
            \infer2[($\land$I)]{X,Y\Yright A\land B}
        \end{prooftree}
        \qquad
        \begin{prooftree}
            \hypo{X\Yright A\land B}
            \hypo{Y(A,B)\Yright C}
            \infer2[($\land$E)]{Y(X)\Yright C}
        \end{prooftree}\\
        \vspace{.2in}
        \begin{prooftree}
            \hypo{X\Yright A}
            \hypo{Y\Yright B}
            \infer2[($\circ$I)]{X;Y\Yright A\circ B}
        \end{prooftree}
        \qquad
        \begin{prooftree}
            \hypo{X\Yright A\circ B}
            \hypo{Y(A;B)\Yright C}
            \infer2[($\circ$E)]{Y(X)\Yright C}
        \end{prooftree}\\
        \vspace{.2in}
        \begin{prooftree}
            \hypo{X\Yright B}
            \hypo{A\Yright\neg B}
            \infer2[($\neg$I)]{X\Yright \neg A}
        \end{prooftree}
        \qquad
        \begin{prooftree}
            \hypo{X\Yright\neg\neg A}
            \infer1[($\neg$E)]{X\Yright A}
        \end{prooftree}\\
        \vspace{.2in}
        \textbf{Structural Rules}\\
        \begin{multicols}{2}
            \begin{prooftree}
                \hypo{X\Yright A}
                \hypo{Y(A)\Yright B}
                \infer2[[Cut]]{Y(X)\Yright B}
            \end{prooftree}
            \columnbreak
            \begin{align*}
                X,(Y,Z) &\leftsquigarrow (X,Y),Z\tag{$\mathsf{eB}$}\\
                X,Y &\leftsquigarrow Y,X\tag{$\mathsf{eC}$}\\
                X,X &\leftsquigarrow X\tag{$\mathsf{eW}$}\\
                X &\leftsquigarrow X,Y\tag{$\mathsf{eK}$}
            \end{align*}
        \end{multicols}
        
    \end{framed}
        \label{fig:cons}
        \caption{Natural deduction rules for $\B$}
    \end{figure}

    In addition to $\mathbf{hB}$, there are three other axiomatic systems we will be interested in. While each of these systems can be presented more efficiently, it suffices to describe them as arising from the addition of axioms and rules to $\mathbf{hB}$. To obtain the first such system, $\mathbf{hDR}^-$, we add the axiom and rule schema below to $\mathbf{hB}$.
    \begin{enumerate}[{A}8)]
        \item $(A\to\neg B)\to(B\to\neg A)$
    \end{enumerate}
    \begin{enumerate}[{R}5)]
        \item $A\Rightarrow\neg(A\to\neg A)$
    \end{enumerate}
    To obtain $\mathbf{hDR}$, we extend $\mathbf{hDR}^-$ with the following axioms and rules.
    \begin{enumerate}[{A}1)]
        \setcounter{enumi}{8}
        \item $((A\to B)\land(B\to C))\to(A\to C)$
        \item $A\lor\neg A$
    \end{enumerate}
    \begin{enumerate}[{R}1)]
        \setcounter{enumi}{7}
        \item $C\lor (A\to B), C\lor A\Rightarrow C\lor B$
        \item $C\lor A\Rightarrow C\lor\neg(A\to\neg A)$
        \item $E\lor(A\to B), E\lor(C\to D)\Rightarrow E\lor((B\to C)\to(A\to D))$
    \end{enumerate}

    For the logic $\mathbf{R}$, we will need both an axiomatic presentation (which we call $\mathbf{hR}$) and a natural deduction presentation (which we call simply $\mathbf{R}$). For the former presentation, we extend $\mathbf{hDR}$ by the following schema additions.

    \begin{enumerate}[{A}1)]
        \setcounter{enumi}{10}
        \item $(A\to B)\to((B\to C)\to(A\to C))$
        \item $(A\to B)\to((C\to A)\to(C\to B))$
        \item $(A\to(A\to B))\to(A\to B)$
        \item $A\to((A\to B)\to B)$
    \end{enumerate}
    
    For the natural deduction presentation of $\R$, we slightly alter $\B$'s negation introduction and tack on a few structural rules as described in Figure 2. That these systems are equivalent is again deferred to~\cite{Read1988}.

    \begin{figure}[htb]
    \begin{framed}
        \centering
        \begin{multicols}{2}
        \textbf{Operational Rule}\\
        \vspace{.2in}
            \begin{prooftree}
                \hypo{X;A\Yright \neg B}
                \hypo{Y\Yright B}
                \infer2[($\neg\mathrm{I}_2$)]{X;Y\Yright \neg A}
            \end{prooftree}
            \columnbreak
            
            \textbf{Structural Rules}\\
            \vspace{-.03in}
            \begin{align*}
                X;(Y;Z) &\leftsquigarrow (X;Y);Z\tag{$\mathsf{B}$}\\
                X;Y &\leftsquigarrow Y;X\tag{$\mathsf{C}$}\\
                X;X &\leftsquigarrow X\tag{$\mathsf{W}$}
            \end{align*}
        \end{multicols}
        
    \end{framed}
        \label{R}
        \caption{Additional natural deduction rules for $\mathbf{R}$}
    \end{figure}

    By observing the axiomatic presentation of these logics, we note that the following obvious containments hold. 
    \begin{displaymath}
        \mathbf{hB}\subseteq\mathbf{hDR}^-\subseteq\mathbf{hDR}\subseteq\mathbf{hR}.
    \end{displaymath}

    Equally clearly, for the natural deduction systems we have that $\mathbf{B}\subseteq\mathbf{R}$. 

    \section{Known Results}

    \begin{mdef}\label{def:depth_tree}
        Where $X$ is a bunch, a \textit{depth-annotated parsing tree for $X$} is a tree whose root is
        \begin{displaymath}
            \underbrace{X}_0
        \end{displaymath}   
        And which is closed under the following extension rules:
        
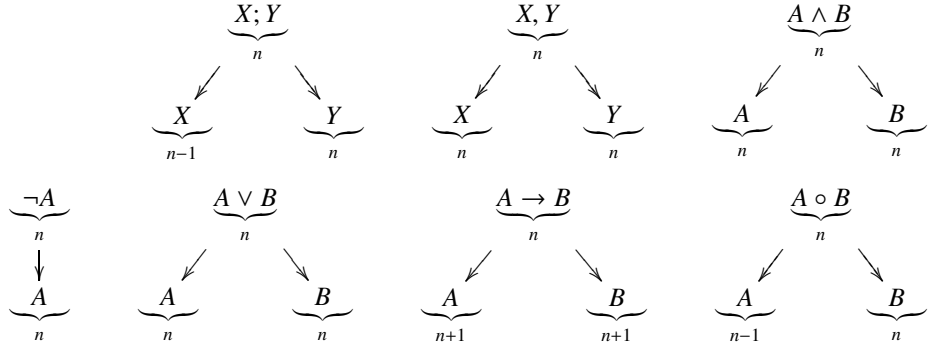
\begin{align*}
            \xymatrix@C=0mm@R=4mm{
                & \underbrace{X;Y}_n\ar[dr]\ar[dl] & \\
                \underbrace{X}_{n-1} & & \underbrace{Y}_{n}
            }
            \qquad
            \xymatrix@C=0mm@R=4mm{
                & \underbrace{X,Y}_n\ar[dr]\ar[dl] & \\
                \underbrace{X}_{n} & & \underbrace{Y}_{n}
            }
            \qquad            
            \xymatrix@C=0mm@R=4mm{
                & \underbrace{A\land B}_n\ar[dr]\ar[dl] & \\
                \underbrace{A}_{n} & & \underbrace{B}_{n}
            }
            \\
            \xymatrix@C=0mm@R=4mm{
                \underbrace{\neg A}_n\ar[d]\\
                \underbrace{A}_{n}
            }
            \qquad
            \xymatrix@C=0mm@R=4mm{
                & \underbrace{A\lor B}_n\ar[dr]\ar[dl] & \\
                \underbrace{A}_{n} & & \underbrace{B}_{n}
            }
            \qquad
            \xymatrix@C=0mm@R=4mm{
                & \underbrace{A\to B}_n\ar[dr]\ar[dl] & \\
                \underbrace{A}_{n+1} & & \underbrace{B}_{n+1}
            }
            \qquad
            \xymatrix@C=0mm@R=4mm{
                & \underbrace{A\circ B}_n\ar[dr]\ar[dl] & \\
                \underbrace{A}_{n-1} & & \underbrace{B}_{n}
            }
        \end{align*}
    \end{mdef}

    \begin{mdef}\label{def:depth}
        If $\underbrace{Y}_n$ occurs in the depth-annotated parsing tree for $X$, then $Y$ is a \textit{subexpression} of $X$, and $n$ is the depth of that occurrence of $Y$ in $X$. If $Y$ is a bunch, then it is a subbunch; if $Y$ is a formula, then it is a subformula.
    \end{mdef}

    Note the lopsided behavior of the fusion and the semicolon: \textit{only} the bunch/formula on left side of---and \textit{not} the bunch/formula on the right side of---a semicolon/fusion changes depth. Also note that in both cases, the depth \textit{decreases}. Thus we can in fact arrive at negative depths (heights?)---e.g. $p$ occurs at depth -1 (height 1?) in $p;q$. All of this is, we grant, a bit odd. None of it is a mistake.
     
    \begin{mdef}\label{def:share}
        Bunches $X$ and $Y$ \textit{share a variable} if some $p\in\At$ is a subexpression of both $X$ and $Y$. They \textit{depth-share a variable} if some $p\in\At$ occurs at depth $n$ in both $X$ and $Y$.
    \end{mdef}

    The classic results about shared variables are the following:
    
    \begin{theorem}[Belnap 1960]\label{th:belnap}
        If $A\to B$ is a theorem of $\mathbf{hR}$ (formulated in $\LL^-$), then $A$ and $B$ share a variable.
    \end{theorem}

    \begin{theorem}[Brady 1984]
        If $A\to B$ is a theorem of $\mathbf{hDR}$ (formulated in $\LL^-$), then $A$ and $B$ depth-share a variable.
    \end{theorem}
    For proofs of the former see \cite{belnap1960} or \cite{andersonbelnap1975}. For a proof of the latter see \cite{brady1984}. Both for sharing and depth-sharing, stronger results (still working only in $\LL^-$) have been proved, most notably in the work of Gemma Robles and Jose Mendez; see e.g.~\cite{robmen2014}.

    In the appendix, we prove a very mildly strengthened version of Theorem~\ref{th:belnap} that will be useful later in the paper.
    \begin{theorem}\label{th:strongervsp}
        If $X\Yright A$ is provable in $\R$, then $X$ and $A$ share a variable.
    \end{theorem}

    Belnap's result, our extension of Belnap's result, Brady's result, \textit{and} Robles and Mendez's extension of Brady's result all rely, essentially, on evaluating logics in cleverly-chosen matrices. In \cite{logan2022}, an alternative way of proving depth-sharing results was presented. We turn to this now, as it is the basis of the results in the remainder of the paper.

    \begin{mdef}\label{def:depth_sub}
        A \textit{depth substitution} is a function $d:\mathbb{Z}\times\At\to\LL$. Given a depth substitution $d$, we extend it to a function (which we also call $d$) $\mathbb{Z}\times\bun\to\bun$ as follows.
        \begin{itemize}
            \item $d^n(\neg A) =\neg d^n(A)$
            \item $d^n(A\lor B)=d^n(A)\lor d^n(B)$
            \item $d^n(A\land B)=d^n(A)\land d^n(B)$
            \item $d^n(A\to B)= d^{n+1}(A)\to d^{n+1}(B)$
            \item $d^n(A\circ B)=d^{n-1}(A)\circ d^n(B)$
            \item $d^n(X,Y)=d^n(X),d^n(Y)$
            \item $d^n(X;Y)=d^{n-1}(X);d^n(Y)$
        \end{itemize}
    \end{mdef}
    Note that depth substitutions \textit{really are} two-place functions. We've just located one of the arguments---the `$n$' one, intuitively---in a funny place; namely in a superscript. We do this for a simple reason: if we don't, then (as thinking about locutions like `$d(X,Y,n)$' makes clear) the comma ends up overloaded.
    
    \begin{mdef}
        A set of bunches $\Gamma$ is \textit{weakly invariant under depth substitutions} when for all depth substitutions $d$, $X\in \Gamma$ only if $d^0(X)\in \Gamma$ as well. $\Gamma$ is \textit{strongly invariant under depth substitutions} when for all depth substitutions $d$ and all $n\in\mathbb{Z}$, $X\in \Gamma$ only if $d^n(X)\in \Gamma$ as well.
    \end{mdef}

    \begin{mdef}\label{def:depth_shifted_function}
        Where $d$ is a depth substitution and $x\in\mathbb{Z}$, we define the depth substitution $d_x$ by saying $d_x^n(p)=d^{n+x}(p)$. 
    \end{mdef}

    \begin{lemma}\label{lem:depth_shift}
		For every bunch $X$, $d_x^n(X)=d^{n+x}(X)$.
	\end{lemma}
	\begin{proof}
		By a straightforward induction on $X$.
	\end{proof}

    \begin{cor}
        $\Gamma$ is weakly invariant under depth substitutions iff $\Gamma$ is strongly invariant under depth substitutions.
    \end{cor}
    
    Given that strong and weak invariance are, per the above corollary, extensionally equivalent, the reader might wonder why the two properties are distinguished at all. The reason, which will make more sense later, is that for stronger syntactic criteria than mere depth, it is often the case that the corresponding distinction in fact marks a genuine difference.
    
    Where $X$ is a bunch and $Y$ is a subexpression of $X$, we write $X(\underline{Y})$ to mean $X$ with a particular occurrence of $Y$ highlighted in some way (perhaps, as is suggested by the notation, by underlining it). Below, we will have need of the following result, whose proof also serves as a nice warmup for some of the results in the next section.

    \begin{lemma}\label{lem:depsemicolon}
        Let $Y$ and $X$ be bunches and $d$ a depth substitution. If the depth of the highlighted occurrence of $X$ in $Y(\underline{X})$ is $c$, then for all $n$, $d^n(Y(\underline{X}))=d^n(Y)(\underline{d^{n+c}(X)})$. In words, the highlighted subbunch of $d^n(Y)$ that corresponds to the depth-$c$ highlighted occurrence of $X$ in $Y(\underline{X})$ is $d^{n+c}(X)$. 
    \end{lemma}
    \begin{proof}
        By induction on the complexity of $Y$. In the base case, $Y$ is a formula and $X$ is a highlighted subformula of $Y$. For this case we need a separate induction. 
        
        In the base case of the internal induction, $Y$ is an atom, $X=Y$, and there's nothing to see. Of the remaining cases, the only interesting ones are the two conditional cases. Since they are essentially the same, we examine only one.
        
        So, suppose that $Y(\underline{X})=A(\underline{D})\to B$. Then $d^n(Y(\underline{X}))=d^{n+1}(A(\underline{D}))\to d^{n+1}(B)$. Note that since the depth of the highlighted occurrence of $D$ in $A(\underline{D})\to B$ is $c$, the depth of the highlighted occurrence of $D$ in $A(\underline{D})$ is $c-1$. So by IH, $d^{n+1}(A(\underline{D}))=d^{n+1}(A)(\underline{d^{n+(c-1)+1}(D)})=d^{n+1}(A)(\underline{d^{n+c}(D)})$. Thus $d^{n+1}(A(\underline{D}))\to d^{n+1}(B)=d^{n+1}(A)(\underline{d^{n+c}(D)})\to d^{n+1}(B)=d^n(A\to B)(\underline{d^{n+c}(D)})$ as required.
        
        We take this to suffice for the inner induction. For the outer induction, the only interesting cases are the semicolon cases. Again we deal with one and leave the other to the reader. 
        
        So, suppose $Y(\underline{X})=Y_1(\underline{X});Y_2$. Since the highlighted occurrence of $X$ is at depth $c$ in $Y$, it occurs at depth $c+1$ in $Y_1$. Thus by the inductive hypothesis, $d^{n-1}(Y_1(\underline{X})) =d^{n-1}(Y_1)(\underline{d^{n+c}(X)})$. Thus
        \begin{align*}
            d^n(Y(\underline{X})) &= d^{n-1}(Y_1(\underline{X}));d^n(Y_2) \\
            &= d^{n-1}(Y_1)(\underline{d^{n+c}(X)});d^n(Y_2) \\
            &= d^{n}(Y)(\underline{d^{n+c}(X)})
        \end{align*}
    \end{proof}

    One reason depth substitutions are interesting is the following pair of observations:
    \begin{theorem}[\cite{logan2022}]\label{th:deth_inv_gives_depth_rel}
        If $X$ is a set of formulas contained in the logic $\mathbf{hR}$ (formulated in $\LL^-$), $X$ is weakly invariant under depth substitutions, and $A\to B\in X$, then $A$ and $B$ depth-share a variable.
    \end{theorem}
    \begin{proof}[Proof Sketch]
        Choose a depth substitution $d$ that is atomic (so its range is a subset of the set of atomic formulas) and injective. Observe that if $A\to B\in X$, then $d^0(A\to B)\in X\subseteq\mathbf{hR}$. But $\mathbf{hR}$ is relevant, so the antecedent and consequent of $d^0(A\to B)$ share a variable. Now trace back through everything to see that it follows from this that $A$ and $B$ must share a variable at the same depth. For more details, see e.g. \cite{logan2022} or \cite{worley2024}.\footnote{It is important to note that the converse of Theorem~\ref{th:deth_inv_gives_depth_rel} is false. Indeed, $\mathbf{hDR}$ is a counterexample. See \cite{logan2023correction} for details.} 
    \end{proof}

    \begin{theorem}\label{thm:dr-isdepsubinv}
        $\mathbf{hDR^-}$ is closed under depth substitutions.
    \end{theorem}
    Clearly these last two theorems together give the result of Theorem 5, albeit for a slightly weaker logic. The reader could be forgiven at this point for wondering what the fuss is all about: if all we get out of doing what we've done so far is a proof of a slightly weaker result, why do the work at all?

    Here's one reason: the proof we get by using the two theorems above is significantly easier and more elegant than the proofs of similar results found in \cite{brady1984} or \cite{logan2021strong}. Here's a second (and better) reason: the route we just gave generalizes rather profoundly, as we'll show in the second half of this paper.

    But there's another sense in which the question misses the whole point. Variable sharing results are neat. It's cool that we can prove them in this different way, and it's cool that this way of proving them is easier and more generalizable. But invariance results like those in Theorem~\ref{thm:dr-isdepsubinv} are fascinating in their own right. They are, we claim, quite striking---not to mention surprising!---results that seem to tell us something deep about those logics to which they apply. 

    \section{Natural Deduction}\label{sec:natdeduct}

    The astute reader will have observed that all of the results cited so far only apply to $\LL^-$. None of them make any mention of $\LL$, let alone $\bun$. It's time we turn to changing this. 

    We will need to introduce a strange sort of object in what follows, for the purpose of which we also need to define the \textit{arity} of a rule to be the total number of consecutions that occur in the rule. Thus, for example, (id) is a unary rule, ($\to$I) is a binary rule and ($\lor$E) is a 4-ary rule.
    \begin{mdef}
        A $\mathbf{B}$-pseudoderivation is a tree $T$ whose nodes are each labeled with either a consecution or a rule name from the set of rule names listed in Figure~\ref{fig:cons} that satisfies the following criteria:
        \begin{itemize}
            \item The root of $T$ is labeled with a consecution.
            \item Every consecution-labeled node is connected only to rule-labeled nodes
            \item Every $n$-ary rule-labeled node is connected to exactly $n$ other nodes, all of which are consecution-labeled.
        \end{itemize}
        We write $\Bpd$ for the set of $\mathbf{B}$-pseudoderivations.
    \end{mdef}
    Intuitively, a $\mathbf{B}$-pseudoderivation is a graph that, from a distance, \textit{looks like} a derivation, but which, on closer inspection, might contain inferential transitions that make no sense at all. We will think about the natural deduction system for $\mathbf{B}$ presented in the previous section in three different ways. Corresponding to these, we identify three natural proof-theoretic objects.

    \begin{mdef}
        $\nd{B}$ is the subset of $\Bpd$ consisting of trees that meet the following criteria:
        \begin{itemize}
            \item At least one rule-labeled node appears in the tree 
            \item Every rule-labeled node, together with the nodes it immediately connects to, forms an instance of the corresponding rule. 
        \end{itemize}
    \end{mdef}
    Members of $\nd{B}$ are called $\mathbf{B}$-derivations. Note that if $T\in\nd{B}$, then all of $T$'s leaf nodes are labeled with either a consecution or with (id). Call the former nodes the \textit{open} leaves of $T$.

    It's useful to see these definitions in action. Thus, while the tree consisting solely of the node containing $A\Yright A\lor B$ is in $\Bpd$, it's not in $\nd{B}$ because it breaks the first additional criterion listed, i.e. every tree in $\nd{B}$ contains at least one rule node. The example below similarly breaks the other criterion, since the instance of ($\lor$E) does not form an instance of the corresponding rule in $\B$.

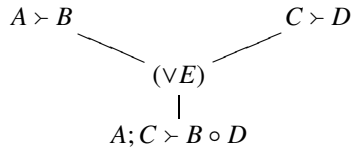
    \begin{center}
        \centerline{\xymatrix@=3mm{
        A\Yright B\ar@{-}[rd]&  & C\Yright D\ar@{-}[ld] \\ 
        & (\lor{E})\ar@{-}[d] & \\ 
        & A;C\Yright B\circ D &
        }}
        
    \end{center}
    
    On the other hand, by replacing the instance of ($\lor$E) with ($\circ$I), we obtain a tree in $\nd{B}$ as well. Its root consists of the consecution $A;C\Yright B\circ D$ and its leaf nodes of the consecutions $A\Yright B$ and $C\Yright D$.

    With $\nd{B}$ on the table we can identify two other sets we will be interested in.
    \begin{mdef}
        $\ndd{B}$ is the set of all pairs $\langle\{X_i\Yright A_i\}_{i=1}^m\mid Y\Yright B\rangle$ for which there is a member of $\nd{B}$ rooted $Y\Yright B$ whose open leaves are labeled with some subset of $\{X_i\Yright A_i\}_{i=1}^m$. Members of $\ndd{B}$ are called $\mathbf{B}$-derivable rules.
    \end{mdef}

    \begin{mdef}
        $\mathbf{B}$ is the set of \textit{provable} consecutions, where $Y\Yright B$ is provable just if \mbox{$\langle\emptyset\mid Y\Yright B\rangle\in\ndd{B}$}.
    \end{mdef}

    \begin{mdef}\label{def:depth_tree_action}
        Given a depth substitution $d$, we extend it (again) to a function $\mathbb{Z}\times\nd{B}\to\Bpd$ as follows, where $R_1\in\{\to\text{E},\circ\text{I}\}$, $R_2\in\{\lor\text{I}_1,\lor\text{I}_2,\neg\text{E},\mathsf{eB},\mathsf{eC},\mathsf{eW},\mathsf{eK}\}$, $R_3\in\{\land\text{I},\neg\text{I}\}$, and $R_4\in\{\land\text{E},\circ\text{E},\mathsf{Cut}\}$ and the depth of the highlighted occurrence of $X$ in $Y(\underline{X})$ is $c$:\footnote{There's no natural way to extend this to a function $\mathbb{Z}\times\Bpd\to\Bpd$. The problem is the highlighting: in CUT, $\lor$E, and $\circ$E, the replacement effected by the rules can serve to highlight the occurrence of $X$ in question. But for arbitrary members of $\Bpd$, that won't be the case.}
        \begin{center}
            \begin{tabular}{c|c|c}
                Rule    &   T  & $d^n(T)$ \\\hline
                (id)    &
                \xymatrix@=3mm{\txt{(id)}\ar@{-}[d] \\ A\Yright A} &
                \xymatrix@=3mm{\txt{(id)}\ar@{-}[d] \\d^n(A)\Yright d^n(A)}
                \\ \hline
                ($\to$I)    &
                \xymatrix@=3mm{S\ar@{-}[d]\\ \txt{($\to$I)}\ar@{-}[d] \\ X\Yright A\to B} &
                \xymatrix@=3mm{d^{n+1}(S)\ar@{-}[d]\\ \txt{($\to$I)}\ar@{-}[d] \\ d^n(X)\Yright d^n(A\to B)}
                \\ \hline
                $R_1$ &
                \xymatrix@=3mm{S_1\ar@{-}[dr] & & S_2\ar@{-}[dl]\\ & \txt{($R_1$)}\ar@{-}[d] & \\ & X\Yright A & } & 
                \xymatrix@=3mm{d^{n-1}(S_1)\ar@{-}[dr] & & d^n(S_2)\ar@{-}[dl]\\ & \txt{($R_1$)}\ar@{-}[d] & \\ & d^n(X)\Yright d^n(A)&}
                \\ \hline
                $R_2$   &
                \xymatrix@=3mm{S\ar@{-}[d] \\ R_2\ar@{-}[d] \\X\Yright A} &
                \xymatrix@=3mm{d^n(S)\ar@{-}[d] \\ R_2\ar@{-}[d] \\d^n(X)\Yright d^n(A)}
                \\ \hline
                $\lor$E &
                \xymatrix@=3mm{S_1\ar@{-}[dr]&S_2\ar@{-}[d]&S_3\ar@{-}[dl]\\&\txt{$\lor$E}\ar@{-}[d]&\\&Y(X)\Yright A&} &
                \xymatrix@=3mm{d^{n+c}(S_1)\ar@{-}[dr]&d^n(S_2)\ar@{-}[d]&d^n(S_3)\ar@{-}[dl]\\&\txt{$\lor$E}\ar@{-}[d]&\\&d^n(Y(X))\Yright d^n(A)&}
                \\ \hline
                $R_3$ &
                \xymatrix@=3mm{S_1\ar@{-}[dr] & & S_2\ar@{-}[dl]\\ & \txt{($R_3$)}\ar@{-}[d] & \\ & X\Yright A & } & 
                \xymatrix@=3mm{d^{n}(S_1)\ar@{-}[dr] & & d^n(S_2)\ar@{-}[dl]\\ & \txt{($R_3$)}\ar@{-}[d] & \\ & d^n(X)\Yright d^n(A)&}
                \\ \hline
                $R_4$ &
                \xymatrix@=3mm{S_1\ar@{-}[dr] & & S_2\ar@{-}[dl]\\ & \txt{($R_4$)}\ar@{-}[d] & \\ & Y(X)\Yright A & } & 
                \xymatrix@=3mm{d^{n+c}(S_1)\ar@{-}[dr] & & d^n(S_2)\ar@{-}[dl]\\ & \txt{($R_4$)}\ar@{-}[d] & \\ & d^n(Y(X))\Yright d^n(A)&}
                \\ \hline
            \end{tabular}
        \end{center}
    \end{mdef}
    
    \begin{theorem}\label{thm:depth_inv}
        $\nd{B}$ is strongly invariant under depth substitutions: if $T\in\nd{B}$ is rooted at $X\Yright A$, $d$ is a depth substitution, and $n\in\mathbb{Z}$, then $d^n(T)\in\nd{B}$ as well and $d^n(T)$ is rooted at $d^n(X)\Yright d^n(A)$.
    \end{theorem}
    \begin{proof}
        By induction on $T$. The (id) case is immediate, and IH immediately finishes the job in the $R_2$, $\lor$E, and $R_3$ cases. We examine one representative from each of the other rows in the above chart and leave the rest to the reader. Note that in each case, if a rule of the appropriate sort labels the last rule node in $T$, then $T$ has the form given in the `$T$' column of the chart. Also note that, by construction, the `rooted' part of the claim is true, so we focus only on the `is a derivation' part. Finally note that by IH, the various $d^i(S_j)$ are all in $\nd{B}$, so all that needs to be shown in each case is that the last rule-labeled-node, together with its neighbors, form an instance of the corresponding rule.
        
        Suppose $T$ has the form given in the ($\to$I)-row of the chart. Then $S$ is rooted at $X;A\Yright B$ and by IH, $d^{n+1}(S)\in\nd{B}$ and $d^{n+1}(S)$ is rooted at $d^{n+1}(X;A)\Yright d^{n+1}(B)$, which is just $d^n(X);d^{n+1}(A)\Yright d^{n+1}(B)$. Since $d^n(A\to B)=d^{n+1}(A)\to d^{n+1}(B)$, it follows that $T$'s last rule node, together with its neighbors, form an instance of the ($\to$I)-rule.
        
        Suppose $T$ has the form given in the ($\to$E)-row of the chart. Then $X$ is of the form $Y;Z$, $S_1$ is rooted at $Y\Yright B\to A$, $S_2$ is rooted at $Z\Yright B$, and $d^n(X)\Yright d^n(A)$ is just $d^{n-1}(Y);d^n(Z)\Yright d^n(A)$. By IH, $d^{n-1}(S_1)\in\nd{B}$ and is rooted at $d^{n-1}(Y)\Yright d^{n-1}(B\to A)$, which is just $d^{n-1}(Y)\Yright d^n(B)\to d^n(A)$. Also by IH, $d^n(S_2)\in\nd{B}$ and is rooted at $d^n(Z)\Yright d^n(B)$. So $T$'s last rule node, together with its neighbors, form an instance of the ($\to$E)-rule.

        Suppose ($\circ$E) labels the last rule node in $T$. Then there are $B$ and $C$ so that $S_1$ is rooted at $X\Yright B\circ C$ and $S_2$ is rooted at $Y(B;C)\Yright A$. By Lemma~\ref{lem:depsemicolon}, $d^n(Y(B;C))\Yright d^n(A)$ is the same as $d^n(Y)(d^{n+c-1}(B);d^{n+c}(C))\Yright d^n(A)$ for some $c\in\mathbb{Z}$. By IH, $d^{n+c}(S_1)\in\nd{B}$ and is rooted at $d^{n+c}(X)\Yright d^{n+c}(B\circ C)$, which is just $d^{n+c}(X)\Yright d^{n+c-1}(B)\circ d^{n+c}(C)$. So the last rule node, together with its neighbors, are in fact an instance of the ($\circ$E)-rule.
        
    \end{proof}

    We end this section by noting the following important corollaries.

    \begin{cor}\label{cor:dep_der_rules}
        Suppose $\langle\{X_i\Yright A_i\}_{i=1}^m\mid Y\Yright B\rangle\in\ndd{B}$, $d$ is a depth substitution, and $n\in\mathbb{Z}$. Then there are $n_i\in\mathbb{Z}$ so that $\langle\{d^{n_i}(X_i)\Yright d^{n_i}(A_i)\}_{i=1}^m\mid d^n(Y)\Yright d^n(B)\rangle$ is in $\ndd{B}$ as well. 
    \end{cor}
    \begin{proof}
        Choose $T\in\nd{B}$ witnessing the derivability of $\langle\{X_i\Yright A_i\}_{i=1}^m\mid Y\Yright B\rangle$. Then by theorem~\ref{thm:depth_inv}, $d^n(T)$ witnesses the derivability of $\langle\{d^{n_i}(X_i)\Yright d^{n_i}(A_i)\}_{i=1}^m\mid d^n(Y)\Yright d^n(B)\rangle$ for some collection of $n_i\in\mathbb{Z}$.
    \end{proof}

    \begin{cor}\label{cor:dep_simple_inv}
        If $X\Yright A\in\mathbf{B}$, $d$ is a depth substitution, and $n\in\mathbb{Z}$, then $d^n(X)\Yright d^n(A)$ is in $\mathbf{B}$ as well. 
    \end{cor}

    \begin{cor}\label{cor:dep_share}
        If $X\Yright A\in\mathbf{B}$, then $X$ and $A$ depth-share a variable.
    \end{cor}
    \begin{proof}
        Exactly as in Theorem~\ref{th:deth_inv_gives_depth_rel}
    \end{proof}

\section{Broader Classes of Substitutions}

    Let's pause to reflect. We began by reciting old results (Theorems~\ref{th:belnap},~\ref{th:deth_inv_gives_depth_rel}, and~\ref{thm:dr-isdepsubinv}). We then extended them: where prior results concerned logics qua sets of sentences, we've now proved analogous results for logics qua sets of derivations in bunched natural deduction systems. But there is another axis of expansion that has been explored; that of tracking more fine-grained syntactic criteria than mere depth. The first such results in this vein were examined in \cite{logan2021strong}. A better extension was given in~\cite{ferguson2023topic}. Here we take the already-extended result just proved and extend it along this other axis in novel ways as well. To begin, let $\seq$ be the set of sequences in $\{l,r,\lambda,\rho,n\}$, and let $\varepsilon$ be the empty sequence. We will use members of $\seq$ in the results below in much the way we used integers in the results above. \textit{However}, for this to work, we must also enforce some reduction rules. To begin, we define \textit{immediate reduction} (written $\squigto'$) to be the relation generated by the following:
    \begin{multicols}{2}
        \begin{enumerate}
            \item\label{reduct:normal_START} $\overline{x}l\lambda\overline{y}\rightsquigarrow'\overline{x}\rho\overline{y}$
            \item $\overline{x}r\lambda\overline{y}\rightsquigarrow'\overline{xy}$
            \item $\overline{x}\lambda r\overline{y}\rightsquigarrow'\overline{xy}$
            \item\label{reduct:normal_END} $\overline{x}\rho r\overline{y}\rightsquigarrow'\overline{x}l\overline{y}$
            \item\label{reduct:dne} $\overline{x}nn\overline{y}\rightsquigarrow'\overline{xy}$ 
        \end{enumerate}
    \end{multicols}
    Informally, these state that occurrences of $l\lambda$ are to be replaced by occurrences of $\rho$, occurrences of $\rho r$ are to be replaced by occurrences of $l$, and occurrences of $r\lambda$, $\lambda r$, and $nn$ are to be removed entirely. Say that a sequence is \textit{reduced} if none of the reduction rules apply to it, and write $\rseq$ for the set of reduced sequences. Finally, say that $\overline{x}$ \textit{reduces to} $\overline{y}$---and write $\overline{x}\squigto\overline{y}$---just if there is a sequence of sequences $\overline{z_1},\dots,\overline{z_n}$ with $n\geq 0$ so that $\overline{x}\squigto'\overline{z_1}\squigto'\dots\squigto'\overline{z_n}\squigto'\overline{y}$.
    
    The following result is proved in Appendix B:
    \begin{theorem}\label{th:red_unique}
        Reductions are unique: if $\overline{x}\squigto\overline{y}\in\rseq$ and $\overline{x}\squigto\overline{z}\in\rseq$, then $\overline{y}=\overline{z}$. 
    \end{theorem}    
    Given this, it makes sense to write $\red(\overline{z})$ for the unique reduced sequence that $\overline{z}$ reduces to. We will need the following results about $\red$ and about $\rseq$; their proofs are also found in Appendix B. 

    \begin{cor}\label{cor:redred}        
        $\red(\overline{zw})=\red(\red(\overline{z})\red(\overline{w}))$
    \end{cor}
    
    \begin{theorem}\label{th:canceling}
        If $\overline{x}$, $\overline{y}$, and $\overline{w}$ are all in $\rseq$, then $\red(\overline{xw})=\red(\overline{yw})$ iff $\overline{x}=\overline{y}$. 
    \end{theorem}
       
    \begin{cor}\label{cor:replace}
        If $\red(\overline{z_1w})=\red(\overline{z_2w})$, then for all $\overline{y}$, $\red(\overline{z_1y})=\red(\overline{z_2y})$. 
    \end{cor}
        
    To put these sequences to work, we first need a way to attach them to subformulas.
    \begin{mdef}[Analogue of Definition~\ref{def:depth_tree}]\label{def:seq_tree}
        Where $X$ is a bunch, an \textit{$\rseq$-annotated parsing tree for $X$} is a tree whose root is
        \begin{displaymath}
            \underbrace{X}_\varepsilon
        \end{displaymath}   
        and which is closed under the following extension rules, where $\overline{x}\in\rseq$.
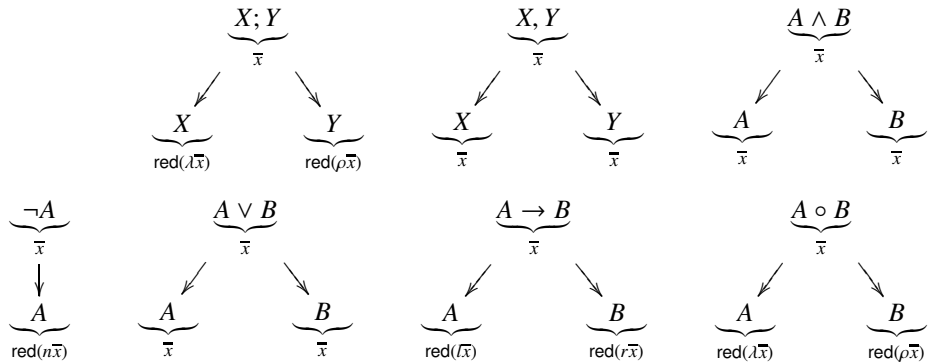
        \begin{align*}
            \xymatrix@C=0mm@R=4mm{
                & \underbrace{X;Y}_{\overline{x}}\ar[dr]\ar[dl] & \\
                \underbrace{X}_{\red(\lambda\overline{x})} & & \underbrace{Y}_{\red(\rho\overline{x})}
            }
            \qquad
            \xymatrix@C=0mm@R=4mm{
                & \underbrace{X,Y}_{\overline{x}}\ar[dr]\ar[dl] & \\
                \underbrace{X}_{\overline{x}} & & \underbrace{Y}_{\overline{x}}
            }
            \qquad            
            \xymatrix@C=0mm@R=4mm{
                & \underbrace{A\land B}_{\overline{x}}\ar[dr]\ar[dl] & \\
                \underbrace{A}_{\overline{x}} & & \underbrace{B}_{\overline{x}}
            }
            \\
            \xymatrix@C=0mm@R=4mm{
                \underbrace{\neg A}_{\overline{x}}\ar[d]\\
                \underbrace{A}_{\red(n\overline{x})}
            }
            \qquad
            \xymatrix@C=0mm@R=4mm{
                & \underbrace{A\lor B}_{\overline{x}}\ar[dr]\ar[dl] & \\
                \underbrace{A}_{\overline{x}} & & \underbrace{B}_{\overline{x}}
            }
            \qquad
            \xymatrix@C=0mm@R=4mm{
                & \underbrace{A\to B}_{\overline{x}}\ar[dr]\ar[dl] & \\
                \underbrace{A}_{\red(l\overline{x})} & & \underbrace{B}_{\red(r\overline{x})}
            }
            \qquad
            \xymatrix@C=0mm@R=4mm{
                & \underbrace{A\circ B}_{\overline{x}}\ar[dr]\ar[dl] & \\
                \underbrace{A}_{\red(\lambda\overline{x})} & & \underbrace{B}_{\red(\rho\overline{x})}
            }
        \end{align*}
    \end{mdef}

    \begin{mdef}[Analogue of Definition~\ref{def:share}]
        Bunches $X$ and $Y$ $\rseq$-share a variable if some $p\in\At$ occurs under the same sequence $\overline{x}$ in both $X$ and $Y$.
    \end{mdef}

    \begin{mdef}\label{def:fallsunder}
        If $\underbrace{Y}_{\overline{x}}$ occurs in the $\rseq$-annotated parsing tree for $X$, then we say this occurrence of $Y$ in $X$ falls under the sequence $\overline{x}$ or, equivalently, that $\overline{x}$ is the sequence under which that occurrence of $Y$ in $X$ falls. 
    \end{mdef}
    
    \begin{mdef}[Analogue of Definition~\ref{def:depth_sub}]
        An $\rseq$-substitution is a function $\rseq\times\At\to\LL$. We extend any such to a function $\rseq\times\bun\to\bun$ as follows:
        \begin{itemize}
            \item $\sigma^{\overline{x}}(A\land B)=\sigma^{\overline{x}}(A)\land \sigma^{\overline{x}}(B)$
            \item $\sigma^{\overline{x}}(A\lor B)=\sigma^{\overline{x}}(A)\lor \sigma^{\overline{x}}(B)$
            \item $\sigma^{\overline{x}}(\neg A)=\neg\sigma^{\red(n\overline{x})}(A)$
            \item $\sigma^{\overline{x}}(A\to B)=\sigma^{\red(l\overline{x})}(A)\to \sigma^{\red(r\overline{x})}(B)$
            \item $\sigma^{\overline{x}}(A\circ B)=\sigma^{\red(\lambda\overline{x})}(A)\circ\sigma^{\red(\rho\overline{x})}(B)$
            \item $\sigma^{\overline{x}}(X,Y)=\sigma^{\overline{x}}(X),\sigma^{\overline{x}}(Y)$
            \item $\sigma^{\overline{x}}(X;Y)=\sigma^{\red(\lambda\overline{x})}(X);\sigma^{\red(\rho\overline{x})}(Y)$
        \end{itemize}
    \end{mdef}

    Recall above that in Definition~\ref{def:depth_tree_action} we extended depth substitutions to act not only on bunches but on derivations. In a moment, we will show how to extend $\rseq$-substitutions in the same way. But it helps to first pause to look at a useful example: Consider the bunch $A\to B;(C\to A;C)$. With minimal effort, it's easy to see that both $A$s occur under $\rho$; both $C$s occur under $\rho\rho$; and that $B$ occurs under $\varepsilon$. It's also quick to check that $A\to B;(C\to A;C)\Yright B$ is in $\mathbf{B}$. Since the $B$ on the right here occurs under $\varepsilon$, we can see $\rseq$-variable sharing in action---the two sides of this consecution $\rseq$-share a variable. 
    
    And there's more: since both $A$s in this consecution occur under the same sequence and both $B$s in this consecution occur under the same sequence and both $C$s in this consecution occur under the same sequence, any $\rseq$-substitution will map this consecution to another member of $\mathbf{B}$. This is the kind of thing we will observe below in greater generality.

    A skeptical reader might play down this behavior by pointing out that, in this case, every $\rseq$-substitution is in fact a uniform substitution. In response to this skepticism, consider instead the consecution $p\to p;p\Yright p$. As before, it's entirely clear that this is in $\mathbf{B}$. But for this consecution, not every $\rseq$-substitution is a uniform substitution. For example, one can apply an $\rseq$-substitution to this consecution and arrive at $p\to q;p\Yright q$. 
    
    But another way to respond to the skeptic is to point out the following: if $C\subseteq D$ are classes of substitutions, $L_C$ is the class of logics invariant under $C$, and $L_D$ the class of logics invariant under $D$, then $L_D\subseteq L_C$. More informally, the \textit{more} substitutions you allow, \textit{fewer} logics will meet the invariance criteria. So what's surprising, given that the class of $\rseq$-substitutions is quite large, is that there are any interesting logics at all that are invariant under $\rseq$-substitutions. That a logic that's not only interesting, but has been the subject of a good deal of study for going on 50 years should be invariant under this class of substitutions is downright flabbergasting. 

    Of course, having said that, we owe the reader a proof. We turn to that now.
   \begin{mdef}[Analogue of Definition~\ref{def:depth_shifted_function}]\label{def:shifted_sequence_function}
        Let $\sigma$ be an $\rseq$-substitution and let $\overline{w}$ and $\overline{y}$ be (not necessarily nonempty) members of $\rseq$. Then we define $\sigma_{\overline{w}\mapsto\overline{y}}:\rseq\times\At\to \LL$ to be the function given by
        \begin{displaymath}
            \sigma^{\overline{x}}_{\overline{w}\mapsto\overline{y}}(p)=\left\{
                    \begin{array}{rl}
                        \sigma^{\red(\overline{zy})}(p) & \text{ if }\overline{zw}\rightsquigarrow\overline{x}\text{ for some }\overline{z} \\
                        \sigma^{\overline{x}}(p) & \text{ otherwise }
                    \end{array}
                \right.
        \end{displaymath}
    \end{mdef}

    \begin{lemma}
        For all $\overline{w}$ and $\overline{y}$ in $\rseq$, $\sigma^{\overline{x}}_{\overline{w}\mapsto\overline{y}}$ is well-defined. 
    \end{lemma}
    \begin{proof}
        Suppose $\overline{z_1w}\squigto\overline{x}$ and $\overline{z_2w}\squigto\overline{x}$. Then $\sigma^{\overline{x}}_{\overline{w}\mapsto\overline{y}}(p)=\sigma^{\red(\overline{z_1y})}(p)$ and $\sigma^{\overline{x}}_{\overline{w}\mapsto\overline{y}}(p)=\sigma^{\red(\overline{z_2y})}(p)$. But since $\overline{z_1w}\squigto\overline{x}$ and $\overline{z_2w}\squigto\overline{x}$, $\red(\overline{z_1w})=\red(\overline{z_2w})$. So by Corollary~\ref{cor:replace}, $\red(\overline{z_1y})=\red(\overline{z_2y})$. Thus $\sigma^{\red(\overline{z_1y})}(p)=\sigma^{\red(\overline{z_2y})}(p)$.
    \end{proof}

    Intuitively, $\sigma_{\overline{w}\mapsto\overline{y}}$ is the $\rseq$-substitution that behaves exactly like $\sigma$ except that it replaces terminal occurrences of $\overline{w}$ with terminal occurrences of $\overline{y}$. A few examples will help make this clear. 
    
    First consider $\sigma_{\varepsilon\mapsto\lambda}$. To calculate e.g. $\sigma^{r}_{\varepsilon\mapsto\lambda}(p)$, we look for a sequence $\overline{z}$ so that $\overline{z}\varepsilon\rightsquigarrow r$. But of course the one-element sequence $r$ is such a $\overline{z}$. So $\sigma_{\varepsilon\mapsto\lambda}^{r}(p)=\sigma^{\red(r\lambda)}(p)=\sigma^\varepsilon(p)$. Next consider $\sigma_{\lambda\mapsto\varepsilon}$. Of particular interest is the computation of $\sigma^{\overline{x}\lambda}_{\varepsilon\mapsto\lambda}(p)$. Here we look for $\overline{z}$ so that $\overline{z}\lambda\rightsquigarrow\overline{x}\lambda$. But of course $\overline{z}=\overline{x}$ will clearly do the job. So $\sigma^{\overline{x}\lambda}_{\lambda\mapsto\varepsilon}(p)=\sigma^{\red(\overline{x}\varepsilon)}(p)=\sigma^{\overline{x}}(p)$. As a trickier case, consider $\sigma^{\rho}_{\lambda\mapsto\varepsilon}(p)$. Since $l\lambda\squigto\rho$, $\sigma^{\rho}_{\lambda\mapsto\varepsilon}(p)=\sigma^{\red(l\varepsilon)}(p)=\sigma^{l}(p)$. Thus, when computing $\sigma_{\lambda\mapsto\varepsilon}$, we `lop off' terminal $\lambda$s, even when they're `hidden' inside $\rho$s. 

    \begin{lemma}[Analogue of Lemma~\ref{lem:depth_shift}]
        For all $X\in\bun$, if $\overline{zw}\rightsquigarrow\overline{x}$, then $\sigma^{\overline{x}}_{\overline{w}\mapsto\overline{y}}(X)=\sigma^{\red(\overline{zy})}(X)$
    \end{lemma}
    \begin{proof}
        By induction on $X$. We consider only the arrow case and leave the rest to the reader. For that case, note that $\sigma^{\overline{x}}_{\overline{w}\mapsto\overline{y}}(A\to B)=\sigma^{\red(l\overline{x})}_{\overline{w}\mapsto\overline{y}}(A)\to\sigma^{\red(r\overline{x})}_{\overline{w}\mapsto\overline{y}}(B)$. But since $\overline{zw}\squigto\overline{x}$, we also have that $l\overline{zw}\squigto l\overline{x}\squigto\red(l\overline{x})$ and $r\overline{zw}\squigto r\overline{x}\squigto\red(r\overline{x})$. Thus by the inductive hypothesis, $\sigma^{\red(l\overline{x})}_{\overline{w}\mapsto\overline{y}}(A)=\sigma^{\red(l\overline{zy})}(A)$ and $\sigma^{\red(r\overline{x})}_{\overline{w}\mapsto\overline{y}}(B)=\sigma^{\red(r\overline{zy})}(A)$. Also, since $\red(l)=l$ and $\red(r)=r$, by Corollary~\ref{cor:redred}, $\red(l\overline{zy})=\red(l\red(\overline{zy}))$ and $\red(r\overline{zy})=\red(r\red(\overline{zy}))$. Altogether then, $\sigma^{\overline{x}}_{\overline{w}\mapsto\overline{y}}(A\to B)=\sigma^{\red(l\red(\overline{zy}))}(A)\to\sigma^{\red(r\red(\overline{zy}))}(B)=\sigma^{\red(\overline{zy})}(A\to B)$ as required.
    \end{proof}
    
    \begin{lemma}[Analogue of Lemma~\ref{lem:depsemicolon}]\label{lem:rseqsubbunches}
        Let $Y$ and $X$ be bunches and $\sigma$ be an $\rseq$-substitution substitution. If the highlighted occurrence of $X$ in $Y(\underline{X})$ falls under $\overline{x}\in\rseq$, then for all $\overline{w}\in\rseq$, $\sigma^{\overline{w}}(Y(\underline{X}))=\sigma^{\overline{w}}(Y)(\underline{\sigma^{\red(\overline{xw})}(X)})$. In words, the subbunch of $\sigma^{\overline{w}}(Y)$ that corresponds to a given occurrence of $X$ in $Y$ that falls under $\overline{x}$ is $\sigma^{\red(\overline{xw})}(X)$. 
    \end{lemma}
    \begin{proof}
        Just as in the proof of Lemma~\ref{lem:depsemicolon}, here too the proof is by induction on the complexity of $Y$. And, just as in that lemma, we will examine only two cases: $Y(\underline{X})=A(\underline{C})\to B$ and $Y(\underline{X})=Y_1;Y_2(\underline{X})$.

        For the first of these, notice first that if the highlighted occurrence of $C$ in $A(\underline{C})\to B$ falls under $\overline{x}$, then $\overline{x}=\red(\overline{y}l)$ where the highlighted occurrence of $C$ in $A(\underline{C})$ falls under $\overline{y}\in\rseq$. Now notice that 
        \begin{align*}
            \sigma^{\overline{w}}(A(\underline{C})\to B) & = \sigma^{\red(l\overline{w})}(A(\underline{C}))\to\sigma^{\red(r\overline{w})}(B)\\
            & =\sigma^{\red(l\overline{w})}(A)(\underline{\sigma^{\red(\overline{y}\red(l\overline{w}))}(C)})\to\sigma^{\red(r\overline{w})}(B) \\
            & = \sigma^{\red(l\overline{w})}(A)(\underline{\sigma^{\red(\overline{y}l\overline{w})}(C)})\to\sigma^{\red(r\overline{w})}(B) \\
            & = 
            \sigma^{\red(l\overline{w})}(A)(\underline{\sigma^{\red(\red(\overline{y}l)\red(\overline{w}))}(C)})\to\sigma^{\red(r\overline{w})}(B) \\
            & = 
            \sigma^{\red(l\overline{w})}(A)(\underline{\sigma^{\red(\overline{x}\overline{w})}(C)})\to\sigma^{\red(r\overline{w})}(B) \\
            & =
            \sigma^{\overline{w}}(A\to B)(\underline{\sigma^{\red(\overline{xw})}(C)})
        \end{align*}
        
        For the second, notice that if the highlighted occurrence of $X$ in $Y_1;Y_2(\underline{X})$ falls under $\overline{x}$, then $\overline{x}=\red(\overline{y}\rho)$ where the highlighted occurrence of $X$ in $Y_2(\underline{X})$ falls under $\overline{y}\in\rseq$. Now notice that
        
        \begin{align*}
            \sigma^{\overline{w}}(Y_1;Y_2(\underline{X}))
            &= \sigma^{\red(\lambda\overline{w})}(Y_1);\sigma^{\red(\rho\overline{w})}(Y_2(\underline{X})) \\
            &= \sigma^{\red(\lambda\overline{w})}(Y_1);\sigma^{\red(\rho\overline{w})}(Y_2)(\underline{\sigma^{\red(\overline{y}\red(\rho\overline{w}))}(X)}) \\
            &= \sigma^{\red(\lambda\overline{w})}(Y_1);\sigma^{\red(\rho\overline{w})}(Y_2)(\underline{\sigma^{\red(\overline{y}\rho\overline{w})}(X)}) \\
            &= \sigma^{\red(\lambda\overline{w})}(Y_1);\sigma^{\red(\rho\overline{w})}(Y_2)(\underline{\sigma^{\red(\red(\overline{y}\rho)\red(\overline{w}))}(X)}) \\
            &= \sigma^{\red(\lambda\overline{w})}(Y_1);\sigma^{\red(\rho\overline{w})}(Y_2)(\underline{\sigma^{\red(\overline{xw})}(X)}) \\
            & = \sigma^{\overline{w}}(Y_1;Y_2)(\underline{\sigma^{\red(\overline{xw})}(X)})
        \end{align*}
    \end{proof}
    The following abbreviations are useful in stating the next result:
    \begin{mdef}\label{def:specialsubs}
        \begin{flalign*}
            \sigma_{\lor I_1},\sigma_{\lor E_2}, \sigma_{\lor E_3}, \sigma_{\land I_1}, \sigma_{\land E_2}, & &\\
            \sigma_{\circ E_2}, \sigma_{\neg I_2}, \sigma_{\neg E_1}  &:= \sigma &\\
            \sigma_{\to I_1} &:= \sigma_{\lambda\mapsto\varepsilon} &\\
            \sigma_{\to\text{E}_1}, \sigma_{\circ I_1} &:= \sigma_{\varepsilon\mapsto\lambda} &\\
            \sigma_{\to\text{E}_2}, \sigma_{\circ I_2} &:= \sigma_{\varepsilon\mapsto\rho} &\\
            \sigma_{\lor E_1}, \sigma_{\land E_1}, \sigma_{\circ E_1} &:= \sigma_{\varepsilon\mapsto\overline{x}}\text{, where the highlighted occurrence of $X$ in $Y(\underline{X})$ occurs under $\overline{x}$} &\\
            \sigma_{\neg I_1} &:= \sigma_{n\mapsto\varepsilon} &\\
        \end{flalign*}
    \end{mdef}

    \begin{mdef}[Analogue of Definition~\ref{def:depth_tree_action}]\label{def:tree_of_formulas}
        Given an $\rseq$-substitution $\sigma$, we extend it (again) to a function $\rseq\times\nd{B}\to\Bpd$ as follows: 
        \begin{itemize}\setlength\itemsep{1ex}
            \item If $T$ is \xymatrix@=3mm{\txt{(id)}\ar@{-}[d] \\ A\Yright A}, then $\sigma^{\overline{x}}(T)$ is \xymatrix@=3mm{\txt{(id)}\ar@{-}[d] \\ \sigma^{\overline{x}}(A)\Yright \sigma^{\overline{x}}A}
            \item For each $n+1$-ary rule R of $\nd{B}$, if $T$ is
            
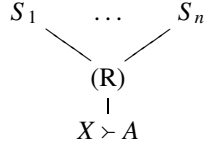
\begin{displaymath}
                \xymatrix@=3mm{
                    S_1\ar@{-}[dr] &  \dots  &  S_n\ar@{-}[dl]\\
                    & \txt{(R)}\ar@{-}[d]   &   \\
                    &   X\Yright A  &
                }
            \end{displaymath}
            Then $\sigma^{\overline{x}}(T)$ is
            
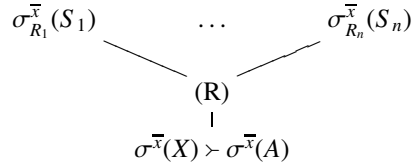
\begin{displaymath}
                \xymatrix@=3mm{
                    \sigma^{\overline{x}}_{R_1}(S_1)\ar@{-}[dr] &  \dots  &  \sigma^{\overline{x}}_{R_n}(S_n)\ar@{-}[dl]\\
                    & \txt{(R)}\ar@{-}[d]   &   \\
                    &   \sigma^{\overline{x}}(X)\Yright \sigma^{\overline{x}}(A)  &
                }
            \end{displaymath}
        \end{itemize}
    \end{mdef}

    \begin{theorem}[Analogue of Theorem~\ref{thm:depth_inv}]\label{thm:rseq_inv}
        $\nd{B}$ is weakly invariant under $\rseq$-substitutions: if $T\in\nd{B}$ is rooted at $X\Yright A$ and $\sigma$ is an $\rseq$-substitution, then $\sigma^\varepsilon(T)\in\nd{B}$ as well and $\sigma^\varepsilon(T)$ is rooted at $\sigma^\varepsilon(X)\Yright \sigma^\varepsilon(A)$.
    \end{theorem}
    \begin{proof}
        As before, the definitions suffice for the second clause---that is, we can see immediately that $\sigma^\varepsilon(T)$ will be rooted at a node labeled by $\sigma^\varepsilon(X)\Yright\sigma^\varepsilon(A)$. We show that $T\in\nd{B}$ by induction on $T$. As before, the only thing that actually needs showing, given the IH, is that the final rule-labeled node will, together with its neighbors, form an instance of the corresponding rule. 
        
        We examine a few illustrative examples and leave the rest to the reader. Throughout, when discussing the case where the last rule node in $T$ is labeled by the $n+1$-ary rule $R$, we assume $T$ has the form
        
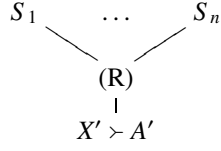
\begin{displaymath}
            \xymatrix@=3mm{
                S_1\ar@{-}[dr]  &   \dots   &   S_n\ar@{-}[dl] \\
                    &   \txt{(R)}\ar@{-}[d] & \\
                    &   X'\Yright A'  &
            }
        \end{displaymath}
        
        \begin{description}
            \item[($\to$I)-case] In this case, $A'=A\to B$ and $S_1$ is rooted at $X;A\Yright B$ and $\sigma^\varepsilon(T)$ is rooted at $\sigma^{\varepsilon}(X)\Yright\sigma^l(A)\to\sigma^r(B)$. By IH, $\sigma^\varepsilon_{\to\text{I}_1}(S_1)\in\nd{B}$ and is rooted at $\sigma^\varepsilon_{\to\text{I}_1}(X;A)\Yright \sigma^\varepsilon_{\to\text{I}_1}(B)$.

            But now observe that $\sigma^\varepsilon_{\to\text{I}_1}(X;A)=\sigma^\lambda_{\lambda\mapsto\varepsilon}(X);\sigma^\rho_{\lambda\mapsto\varepsilon}(A)=\sigma^\varepsilon(X);\sigma^l(A)$ and $\sigma^\varepsilon_{\to\text{I}_1}(B)=\sigma^r(B)$. Thus $\sigma^\varepsilon_{\to\text{I}_I}(S_1)$ is rooted at $\sigma^\varepsilon(X);\sigma^l(A)\Yright\sigma^r(B)$, and thus $\sigma^\varepsilon(T)\in\nd{B}$. 

            \item[($\to$E)-case] In this case $X'=X;Y$, $A'=B$, $S_1$ is rooted at $X\Yright A\to B$ for some $A$, $S_2$ is rooted at $Y\Yright A$, and $\sigma^\varepsilon(T)$ is rooted at $\sigma^\lambda(X);\sigma^\rho(Y)\Yright\sigma^\varepsilon(B)$. By IH, $\sigma^\varepsilon_{\to\text{E}_1}(S_1)\in\nd{B}$ and is rooted at $\sigma^\varepsilon_{\to\text{E}_1}(X)\Yright \sigma^\varepsilon_{\to\text{E}_1}(A\to B)$, and $\sigma^\varepsilon_{\to\text{E}_2}(S_2)\in\nd{B}$ and is rooted at $\sigma^\varepsilon_{\to\text{E}_2}(Y)\Yright\sigma^\varepsilon_{\to\text{E}_2}(A)$.

            With a bit of effort, one can now observe that 
            $\sigma^{\varepsilon}_{\to\text{E}_1}(X)=\sigma^\lambda(X)$, that $\sigma^{l}_{\to\text{E}_1}(A)=\sigma^{\rho}(A)$, and that $\sigma^{r}_{\to\text{E}_1}(B)=\sigma^{\varepsilon}(B)$. Thus $\sigma^\varepsilon_{\to\text{E}_1}(S_1)$ is rooted at $\sigma^\lambda(X)\Yright\sigma^\rho(A)\to\sigma^\varepsilon(B)$.
            
            In a similar way, one sees that $\sigma^{\varepsilon}_{\to\text{E}_2}(Y)=\sigma^{\rho}(Y)$ and that $\sigma^{\varepsilon}_{\to\text{E}_2}(A)= \sigma^{\rho}(A)$. Thus $\sigma^{\varepsilon}_{\to\text{E}_2}(S_2)$ is rooted at $\sigma^\rho(Y)\Yright\sigma^\rho(A)$. So $\sigma^\varepsilon(T)\in\nd{B}$.
            \item[($\circ$E)-case] In this case, $X'=Y(X)$, $A'=C$, $S_1$ is rooted at $X\Yright A\circ B$, $S_2$ is rooted at $Y(A;B)\Yright C$, and $\sigma^\varepsilon(T)$ is rooted at $\sigma^\varepsilon(Y(X))\Yright\sigma^\varepsilon(C)$. 

            By Lemma~\ref{lem:rseqsubbunches}, since the highlighted occurrence of $X$ in $Y(\underline{X})$ occurs under $\overline{x}$, we know that $\sigma^{\varepsilon}(Y(\underline{X}))=\sigma^{\varepsilon}(Y)(\underline{\sigma^{\overline{x}}(X)})$. By the same reasoning, $\sigma^{\varepsilon}_{\circ E_2}(Y(\underline{A;B}))=\sigma^{\varepsilon}_{\circ E_2}(Y)(\underline{\sigma^{\overline{x}}_{\circ E_2}(A;B)})$, and thus $\sigma^{\varepsilon}_{\circ E_2}(Y(\underline{A;B}))=\sigma^{\varepsilon}_{\circ E_2}(Y)(\underline{\sigma^{\lambda\overline{x}}_{\circ E_2}(A);\sigma^{\rho\overline{x}}_{\circ E_2}(B)})$. 

            Again, shifting all the superscripts takes a bit of effort, but the results are the following. First, on the $S_1$ side we have $\sigma^\varepsilon_{\circ \text{E}_1}(X)=\sigma^{\overline{x}}(X)$, $\sigma^\lambda_{\circ \text{E}_1}{(A)}=\sigma^{\lambda\overline{x}}(A)$, and $\sigma^\rho_{\circ \text{E}_1}(B)=\sigma^{\rho\overline{x}}(B)$. So $\sigma^\varepsilon_{\circ \text{E}_1}(S_1)$ is rooted at $\sigma^{\overline{x}}(X)\Yright\sigma^{\lambda\overline{x}}(A)\circ\sigma^{\rho\overline{x}}(B)$. 

            Then, on the $S_2$ side we see that $\sigma^{\varepsilon}_{\circ E_2}(Y(A;B))=\sigma^{\varepsilon}(Y)(\sigma^{\lambda\overline{x}}(A);\sigma^{\rho\overline{x}}(B))$ and $\sigma^{\varepsilon}_{\circ E_2}(C)= \sigma^{\varepsilon}(C)$. So $\sigma^\varepsilon_{\circ \text{E}_2}(S_2)$ is rooted at $\sigma^{\varepsilon}(Y)(\sigma^{\lambda\overline{x}}(A);\sigma^{\rho\overline{x}}(B))\Yright\sigma^{\varepsilon}(C)$.

            It follows now that $\sigma^\varepsilon(T)\in\nd{B}$ as required.
        \end{description}    
    \end{proof}

    We note that while Theorem~\ref{thm:rseq_inv} is an analogue of Theorem~\ref{thm:depth_inv}, there is an important difference: Theorem~\ref{thm:depth_inv} proves \textit{strong} invariance---for all depth substitutions $d$ \textbf{and all $n\in\mathbb{Z}$}, if $T\in\nd{B}$, then $d^n(T)\in\nd(B)$. Theorem~\ref{thm:rseq_inv}, on the other hand, proves only \textit{weak} invariance. 
    
    As it turns out, the strong invariant analogue of Theorem~\ref{thm:rseq_inv} isn't just hard to prove---it's actually false. Examining nearly any derivation containing a piece of intensional vocabulary will make this apparent, so we leave it to the reader to check a case they find illuminating.

    Lastly, we note that just as with Theorem~\ref{thm:depth_inv}, Theorem~\ref{thm:rseq_inv} has some notable corollaries. First, we need a couple definitions:
    \begin{mdef}
        Let $\seq*$ be the set of sequences in the alphabet $\{l,r,\lambda,\rho,n,\varepsilon\}$, and let $\varepsilon^*$ be the empty $\seq^*$.
    \end{mdef}

    \begin{mdef}
        Where $\overline{x}$ and $\overline{y}$ are both members of $\seq^*$ of length $n$, we write $\sigma_{\overline{x}\mapsto\overline{y}}$ for $(\dots(\sigma_{x_1\mapsto y_1})\dots)_{x_n\mapsto y_n}$
    \end{mdef}

    \begin{cor}[Analogue of Corollary~\ref{cor:dep_der_rules}]
        Suppose $\langle\{X_i\Yright A_i\}_{i=1}^m\mid Y\Yright B\rangle\in\ndd{B}$ and $\sigma$ is an $\rseq$-substitution. Then for each $1\leq i\leq m$ there are $\overline{x^i}\in\seq^*$ and $\overline{y^i}\in\seq^*$ so that $\langle\{\sigma^\varepsilon_{\overline{x^i}\mapsto\overline{y^i}}(X_i)\Yright \sigma^\varepsilon_{\overline{x^i}\mapsto\overline{y^i}}(A_i)\}_{i=1}^m\mid \sigma^\varepsilon(Y)\Yright \sigma^\varepsilon(B)\rangle\in\ndd{B}$.
    \end{cor}

    \begin{cor}[Analogue of Corollary~\ref{cor:dep_simple_inv}]
        If $X\Yright A\in\mathbf{B}$ and $\sigma$ is a depth substitution, then $\sigma^\varepsilon(X)\Yright\sigma^\varepsilon(A)$ is in $\mathbf{B}$ as well. 
    \end{cor}

    \begin{cor}[Analogue of Corollary~\ref{cor:dep_share}]
        If $X\Yright A\in\mathbf{B}$, then $X$ and $A$ $\rseq$-share a variable.
    \end{cor}
    \begin{proof}
        Let $i$ be an atomic injective $\rseq$-substitution. By the previous corollary, $i^\varepsilon(X)\Yright i^\varepsilon(A)\in\mathbf{B}$ as well. Since $\mathbf{B}\subseteq\mathbf{R}$, by the extension of Belnap's variable sharing result proved in the appendix, $i^\varepsilon(X)$ and $i^\varepsilon(A)$ share a variable. But then since $i$ is atomic injective, $X$ and $A$ must $\rseq$-share a variable.
    \end{proof}
 
    \section{Discussion}

    There is a sense---a sense we do not know how to express quite as concretely as we would like---in which $\rseq$-hyperformalism seems to be the strongest plausible flavor of hyperformalism on offer for the language we've considered in this paper. The loose idea is this: no plausible sort of hyperformalism ought to track where formulas occur in \textit{extensional} contexts. And $\rseq$-hyperformalism tracks, for the intensional parts of the language being considered, everything one might plausibly track: not only \textit{whether} but also \textit{how} each atom occurs in the scope of each intensional operator.

    At any rate, whether $\rseq$-hyperformalism is the strongest sort of hyperformalism or not, it's certainly \textit{quite strong}. It can also be motivated along topical lines as described in the introduction.

    What we've shown in this paper is that $\rseq$-hyperformalism is not an accidental feature of the logic $\mathbf{B}$. It is instead something that $\mathbf{B}$ inherits from the structure of its class of derivations which also enjoys a version of $\rseq$-hyperformalism. 

    Finally, we note that there are four obvious ways to extend the project here. First: one can add reduction rules to the ones considered here and attempt to thereby develop novel flavors of hyperformalism for logics other than $\mathbf{B}$. Second: one can examine logics weaker than $\mathbf{B}$. One natural option---explored in \cite{ferguson2023topic} is the logic $\mathbf{BM}$. Third: one can (see e.g. \cite{Restall2000} for details) augment the calculus we gave for $\B$ above with lambda terms. Lambda terms \textit{just are} names of proofs, which are in turn just derivations. We've seen that $\rseq$-substitutions act naturally on the last of these, so there's a natural way for them to act on lambda terms. The details here should be worked out. Fourth: for those logics in this area that admit a normalization theorem, the proofs in fact form a category. So we actually have a natural way for $\rseq$-substitutions to act on a certain type of category. Details here also need to be worked out as well. We are of the impression that these last two inroads are the most promising way forward for research on hyperformalism. 

\newpage
\appendix

\section{Proof that \textbf{R} has the strong variable sharing property}

Following \cite{Read1988}, for each bunch $X$, define the characteristic formula\footnote{Read restricts his definition to \textit{finite} bunches. We don't explicitly do so because, as we've defined them in this paper, all bunches are already finite.} of $X$, $\cf(X)$ as follows:
\begin{itemize}
    \item If $X=A$ is an atomic bunch (that is, a formula), then $\cf(X)=X$.
    \item If $X=Y_1,Y_2$, then $\cf(X)=\cf(Y_1)\land\cf(Y_2)$.
    \item If $X=Y_1;Y_2$, then $\cf(X)=\cf(Y_1)\circ\cf(Y_2)$.
\end{itemize}

It follows from Proposition 4.3 in \cite{Read1988} that in $\mathbf{R}$ we have the following:
\begin{itemize}
    \item $X\vdash_{\mathbf{R}}\cf(X)$ 
    \item If $Y(X)\vdash A$, then $Y(\cf(X))\vdash A$.
\end{itemize}

We now define the translation function $\tau$ as follows:
\begin{itemize}
    \item If $A$ is atomic, then $\tau(A)=A$.
    \item $\tau(A\land B)=\tau(A)\land\tau(B)$.
    \item $\tau(A\lor B)=\tau(A)\lor\tau(B)$.
    \item $\tau(\neg A)=\neg\tau(A)$.
    \item $\tau(A\to B)=\tau(A)\to\tau(B)$.
    \item $\tau(A\circ B)=\neg(\tau(A)\to\neg\tau(B))$.
    \item If $X$ is a bunch, then $\tau(X)=\tau(\cf(X))$.
\end{itemize}

\begin{lemma}
    For all formulas $A$, $\tau(A)\vdash_\mathbf{R} A$. 
\end{lemma}
\begin{proof}
    By induction on $A$. The only case worth discussing is the fusion case, for which we first observe the following:    
    
    \begin{prooftree}
            \hypo{\neg(A\to\neg B)\vdash_{\mathbf{R}}\neg(A\to\neg B)}
            \hypo{A\vdash_{\mathbf{R}}\tau(A)}
            \hypo{B\vdash_{\mathbf{R}}\tau(B)}            \infer2{A;B\vdash_{\mathbf{R}}\tau(A)\circ\tau(B)}
            \hypo{\neg(A\circ B)\vdash_{\mathbf{R}}\neg(A\circ B)}
            \infer2{A;\neg(A\circ B)\vdash_{\mathbf{R}}\neg B}
            \infer1{\neg(A\circ B);A\vdash_{\mathbf{R}}\neg B}
            \infer1{\neg(A\circ B)\vdash_{\mathbf{R}} A\to\neg B}
            \hypo{\neg(A\to\neg B)\vdash_{\mathbf{R}}\neg(A\to\neg B)}
            \infer2{\neg(A\circ B)\vdash_{\mathbf{R}}\neg\neg(A\to\neg B)}
            \infer2{\neg(A\to\neg B)\vdash_{\mathbf{R}}\neg\neg(A\circ B)}
            \infer1{\neg(A\to\neg B)\vdash_{\mathbf{R}} A\circ B}
    \end{prooftree}\bigskip
    
    Thus, in particular, and using the inductive hypothesis, we have\bigskip
    
    \begin{prooftree}
        \hypo{\neg(\tau(A)\to\neg\tau(B))\vdash_{\mathbf{R}}\tau(A)\circ\tau(B)}
        \hypo{\tau(A)\vdash_{\mathbf{R}} A}
        \hypo{\tau(B)\vdash_{\mathbf{R}} B}
        \infer2{\tau(A);\tau(B)\vdash_{\mathbf{R}} A\circ B}
        \infer2{\neg(\tau(A)\to\neg\tau(B))\vdash_{\mathbf{R}}A\circ B}
    \end{prooftree}
\end{proof}

\begin{lemma}\label{lem:translationsobeynl}
    Suppose $X$ and $Y$ are bunches and $A\in\LL$.
    \begin{enumerate}[i.]
        \item $X\vdash_{\mathbf{R}}\tau(X)$
        \item If $Y(X)\vdash_{\mathbf{R}} A$, then $Y(\tau(X))\vdash A$
    \end{enumerate}
\end{lemma}
\begin{proof}
    By simultaneous induction on the complexity of $X$ in both parts of the lemma. In the base case of the induction, $X$ is an atomic formula, and both (i) and (ii) are immediate. In the inductive step, we must construct a natural deduction proof for each outermost connective. In the case of fusion, we defer the proof to a well-known result, which can be shown with mild elbow grease: $A\circ B\vdash_{\R}\neg(A\to\neg B)$ and if $Y(A\circ B)\vdash_{\R} C$, then $Y(\neg(A\to\neg B))\vdash_{\R} C$. In fact, the former bit of elbow grease will be fleshed out in a case below. We outline only two cases in full and leave the rest to the reader.

    First, suppose $X=A\to B$.
    \begin{enumerate}[i.]
        \item We want to show $A\to B\vdash_{\R} \tau(A)\to\tau(B)$. By the previous lemma, we can fill in the dots on the following derivation, which does the job:
        \begin{center}
            \begin{prooftree}
                \hypo{A\to B\Yright A\to B}
                \hypo{\vdots}
                \infer1{\tau(A)\Yright A}
                \infer2{A\to B;\tau(A)\Yright B}
                \hypo{\vdots}
                \infer1[IH]{B\Yright \tau(B)}
                \infer2[Cut]{A\to B;\tau(A)\Yright\tau(B)}
                \infer1{A\to B\Yright\tau(A)\to\tau(B)}
            \end{prooftree}
        \end{center}

        \item Suppose $Y(A\to B)\vdash_{\R} C$. By the previous lemma, $\tau(A)\to\tau(B)\vdash_{\R} A\to B$, so we can use CUT to immediately obtain $Y(\tau(A\to B))\vdash_{\R} C$.
    \end{enumerate}

    \noindent Suppose $X=X_1;X_2$. 
    \begin{enumerate}[i.]
        \item
            By IH, $X_1\vdash_{\mathbf{R}}\tau(X_1)=\tau(\cf(X_1))$ and $X_2\vdash_{\mathbf{R}}\tau(X_2)=\tau(\cf(X_2))$. Thus $X_1;X_2\vdash_{\mathbf{R}}\tau(\cf(X_1))\circ\tau(\cf(X_2))$. But for any two formulas $A$ and $B$, $A\circ B\vdash_{\mathbf{R}}\neg(A\to\neg B)$, as one can see from the following derivation:
            \begin{center}
                \begin{prooftree}
                    \hypo{A\circ B\vdash_{\mathbf{R}} A\circ B}
                    \hypo{A\to\neg B\vdash_{\mathbf{R}}A\to\neg B}
                    \hypo{A\vdash_{\mathbf{R}} A}
                    \infer2{A\to\neg B;A\vdash_{\mathbf{R}} \neg B}
                    \infer1{A;A\to\neg B\vdash_{\mathbf{R}}\neg B}
                    \hypo{B\vdash_{\mathbf{R}}B}
                    \infer2{A;B\vdash_{\mathbf{R}}\neg(A\to\neg B)}
                    \infer2{A\circ B\vdash_{\mathbf{R}}\neg(A\to\neg B)}
                \end{prooftree}
            \end{center}
            Thus, in particular, $\tau(\cf(X_1))\circ\tau(\cf(X_2))\vdash_{\mathbf{R}}\neg(\tau(\cf(X_1))\to\neg\tau(\cf(X_2)))$. So by an application of CUT, $X_1;X_2\vdash_{\mathbf{R}}\neg(\tau(\cf(X_1))\to\neg\tau(\cf(X_2)))=\tau(X_1;X_2)$.

        \item Suppose $Y(X_1;X_2)\vdash_{\R} A$. Then by two applications of induction hypothesis, $Y(\tau(X_1);\tau(X_2))\vdash_{\R} C$. It can be easily shown from this that $Y(\tau(X_1)\circ\tau(X_2))\vdash_{\R} C$. But as we saw in the proof of the preceding lemma, $\neg(\tau(X_1)\to\neg\tau(X_2))\vdash_{\mathbf{R}}\tau(X_1)\circ\tau(X_2)$. Thus by an application of CUT, $Y(\cf(X_1;X_2))\vdash_{\R} C$.
    \end{enumerate}
\end{proof}

Before proving our next result, we leave it to the reader to verify the following facts about $\mathbf{hR}$:
\begin{enumerate}[{Fact} 1.]
    \item If $\neg A\to B\in\mathbf{hR}$, then $\neg B\to A\in\mathbf{hR}$.\label{contrapose1}
    \item If $A\to(B\to C)\in\mathbf{hR}$, then $B\to(A\to C)\in\mathbf{hR}$.\label{permutation}
    \item If $A\to(\neg B\to\neg C)\in\mathbf{hR}$, then $A\to(C\to B)\in\mathbf{hR}$.\label{contrapose2}
    \item If $A\to B\in\mathbf{hR}$, then for all $C$, $(A\land C)\to(B\land C)\in\mathbf{hR}$.
    \item If $A\to B\in\mathbf{hR}$, then for all $C$, $(C\land A)\to(C\land B)\in\mathbf{hR}$.
    \item If $A\to B\in\mathbf{hR}$, then for all $C$, $\neg(A\to\neg C)\to\neg(B\to\neg C)\in\mathbf{hR}$.
    \item If $A\to B\in\mathbf{hR}$, then for all $C$, $\neg(C\to\neg A)\to\neg(C\to\neg B)\in\mathbf{hR}$.
    \item If $A\to B\in\mathbf{hR}$ and $B\to C\in\mathbf{hR}$, then $A\to C\in\mathbf{hR}$. 
\end{enumerate}

\begin{lemma}\label{substitution}
    If $\tau(A)\to\tau(Z)\in\mathbf{hR}$ and $Z$ is a subbunch of $Y$, then $\tau(Y(A))\to\tau(Y(Z))\in\mathbf{hR}$.
\end{lemma}
\begin{proof}
    By induction on Y. The base case is immediate and the remaining four cases are handled by Facts 4-7.
\end{proof}

\begin{lemma}\label{lem:r_to_hr}
    If $X\vdash_{\mathbf{R}} A$, then $\tau(X)\to\tau(A)\in\mathbf{hR}$.
\end{lemma}
\begin{proof}
    By induction on $\vdash$. In the base case, $X=A$ and it suffices to observe that $\tau(A)\to\tau(A)$ is an instance of an axiom of $\mathbf{hR}$.

    Suppose the last rule was ($\to$I). Then by IH, $\neg(\tau(X)\to\neg\tau(A))\to\tau(B)\in\mathbf{hR}$. From here we reason as follows:
    \begin{align}
        \neg(\tau(X)\to\neg\tau(A))\to\tau(B)\in\mathbf{hR}
        & \text{ by IH} \\
        \neg\tau(B)\to(\tau(X)\to\neg\tau(A))\in\mathbf{hR}
        & \text{ from (1) by Fact~\ref{contrapose1}} \\
        \tau(X)\to(\neg\tau(B)\to\neg\tau(A))\in\mathbf{hR}
        & \text{ from (2) by Fact~\ref{permutation}} \\
        \tau(X)\to(\tau(A)\to\tau(B))\in\mathbf{hR}
        & \text{ from (3) by Fact~\ref{contrapose2}}
    \end{align}

    Suppose the last rule was ($\to$E). Then by IH, $\tau(X)\to(\tau(A)\to\tau(B))\in\mathbf{hR}$ and $\tau(Y)\to\tau(A)\in\mathbf{hR}$. So $\tau(A)\to(\tau(X)\to\tau(B))$ is a theorem of $\mathbf{hR}$ by Fact 2. But since $\tau(Y)\to\tau(A)$ is a theorem of $\mathbf{hR}$, so in fact $\tau(Y)\to(\tau(X)\to\tau(B))$ is a theorem of $\mathbf{hR}$. It follows that $\tau(Y)\to(\neg\tau(B)\to\neg\tau(X))$ is a theorem of $\mathbf{hR}$, and thus so is $\neg\tau(B)\to(\tau(Y)\to\neg\tau(X))$. From here we get that $\neg\tau(B)\to(\tau(X)\to\neg\tau(Y))$ is a theorem of $\mathbf{hR}$, and thus that $\neg(\tau(X)\to\neg\tau(Y))\to\tau(B))$ is as well, as required.

    Suppose the last rule was ($\circ$E). Then by IH, $\tau(X)\to\tau(A\circ B)\in\mathbf{hR}$ and $\tau(Y(A;B))\to\tau(C)\in\mathbf{hR}$. But $\tau(Y(A;B))=\tau(Y)(\tau(A;B))=\tau(Y)(\tau(A\circ B))=\tau(Y(A\circ B))$. Thus by Lemma~\ref{substitution}, $\tau(Y(X))\to\tau(Y(A\circ B))\in\mathbf{hR}$. So By Fact 8, $\tau(Y(X))\to\tau(C)\in\mathbf{hR}$.
\end{proof}

From this, it's nearly immediate to get that where we want to go:\bigskip

\noindent\textbf{Theorem~\ref{th:strongervsp}.}
    \emph{If $X\vdash_{\mathbf{R}} A$, then $X$ and $A$ share a variable.}
\begin{proof}
    If $X\vdash_{\mathbf{R}} A$, then $\tau(X)\to\tau(A)\in\mathbf{hR}$. So $\tau(X)$ and $\tau(A)$ share a variable. But by definition, for all $Y$, a variable occurs in $\tau(Y)$ iff it occurs in $Y$. So $X$ and $A$ share a variable.
\end{proof}

\section{Proofs of Facts About Reduced Sequences}

    We note that our proof of Theorem~\ref{th:red_unique} borrows heavily from the discussion in \cite{selinger2013lecturenoteslambdacalculus} concerning the proof of the Church-Rosser Theorem. There as here, the first thing to do is prove the following `one-step' lemma:

    \begin{lemma}\label{lemma:one_step}
        If $\overline{x}\squigto'\overline{w_1}$ and $\overline{x}\squigto'\overline{w_2}$, then there is $\overline{w_3}$ so that $\overline{w_1}\squigto'\overline{w_3}$ and $\overline{w_2}\squigto'\overline{w_3}$.
    \end{lemma}
    \begin{proof}
        First, suppose the reductions $\overline{x}\squigto'\overline{w_1}$ and $\overline{x}\squigto'\overline{w_2}$ act on disjoint parts of $\overline{x}$. That is suppose that $\overline{x}=\overline{s}x_{i}x_{i+1}\overline{t}x_{j}x_{j+1}\overline{u}$, $i+1< j$, and $\overline{w_1}=\overline{s}y\overline{t}x_{j}x_{j+1}\overline{u}$ while $\overline{w_2}=\overline{s}x_{i}x_{i+1}\overline{t}z\overline{u}$. Then clearly if $\overline{w_3}=\overline{s}y\overline{t}z\overline{u}$, we have that $\overline{w_1}\squigto'\overline{w_3}$ and $\overline{w_2}\squigto'\overline{w_3}$. 

        What remains are the cases where the reductions act on overlapping parts of $\overline{x}$, that is, where $\overline{x}=\overline{s}x_{i}x_{i+1}x_{i+2}\overline{t}$ and $\overline{w_1}=\overline{s}yx_{i+2}\overline{t}$ while $\overline{w_2}=\overline{s}x_iz\overline{t}$. There are five options for the subsequence $x_ix_{i+1}x_{i+2}$ for which this can happen:
        \begin{itemize}
            \item $l\lambda r$
            \item $r\lambda r$
            \item $\lambda r\lambda$
            \item $\rho r\lambda$
            \item $nnn$
        \end{itemize}
        We examine each case in turn.
        \begin{description}
            \item[$l\lambda r$-case] In this case, $\overline{w_1}=\overline{s}\rho r\overline{t}$ and $\overline{w_2}=\overline{s}l\overline{t}$. But then if $\overline{w_3}=\overline{w_2}$ we immediately have that $\overline{w_1}\squigto'\overline{w_3}$ and $\overline{w_2}\squigto'\overline{w_3}$.
            \item[$r\lambda r$-case] In this case, $\overline{w_1}=\overline{w_2}=\overline{s} r\overline{t}$ and we are immediately done.
            \item[$\lambda r\lambda$-case] In this case, $\overline{w_1}=\overline{w_2}=\overline{s} \lambda\overline{t}$ and we are immediately done.
            \item[$\rho r\lambda$-case] In this case, $\overline{w_1}=\overline{s}l\lambda\overline{t}$ and $\overline{w_2}=\overline{s}\rho\overline{t}$. But then if $\overline{w_3}=\overline{w_2}$ we immediately have that $\overline{w_1}\squigto'\overline{w_3}$ and $\overline{w_2}\squigto'\overline{w_3}$.
            \item[$nnn$-case] In this case, $\overline{w_1}=\overline{w_2}=\overline{s}n\overline{t}$ and we are immediately done.
        \end{description}
    \end{proof}
    
    \begin{proof}[Proof of Theorem~\ref{th:red_unique}]
        Let $\overline{x}\squigto'\overline{y_1}\squigto'\dots\squigto'\overline{y_{n}}\squigto'\overline{y}$ and $\overline{x}\squigto'\overline{z_1}\squigto'\dots\squigto'\overline{z_{m}}\squigto'\overline{z}$. We show by induction on $n+m$ that there is $\overline{w}$ so that $\overline{y}\squigto'\overline{u_1}\squigto'\dots\squigto'\overline{u_m}\squigto'\overline{w}$ and $\overline{z}\squigto'\overline{v_1}\squigto'\dots\squigto'\overline{v_n}\squigto'\overline{w}$.

        In words: if we can get from $\overline{x}$ to $\overline{y}$ in $n$ steps and we can get from $\overline{x}$ to $\overline{z}$ in $m$ steps, then there is $\overline{w}$ so that we can get from $\overline{y}$ to $\overline{w}$ in $m$ steps and we can get from $\overline{z}$ to $\overline{w}$ in $n$ steps.

        In the base case, $n=m=0$, so in fact $\overline{x}\squigto'\overline{y}$ and $\overline{x}\squigto'\overline{z}$. Here the result follows by Lemma~\ref{lemma:one_step}. 

        Now suppose that we have the result for $n+m\leq k$ and let $n+m=k+1$. Let the reduction sequences be $\overline{x}\squigto'\overline{y_1}\squigto'\dots\squigto'\overline{y_{n}}\squigto'\overline{y}$ and $\overline{x}\squigto'\overline{z_1}\squigto'\dots\squigto'\overline{z_{m}}\squigto'\overline{z}$.

        By IH, there is $\overline{w_1}$ so that $\overline{y_n}\squigto'\overline{u_1}\squigto'\dots\squigto'\overline{u_m}\squigto'\overline{w_1}$ and $\overline{z}\squigto'\overline{v_1}\squigto'\dots\squigto'\overline{v_{n-1}}\squigto'\overline{w_1}$. But then since $\overline{y_n}\squigto'\overline{y}$, IH also gives that there is $w_2$ so that $\overline{w_1}\squigto'\overline{w_2}$ and $\overline{y}\squigto'\overline{t_1}\squigto'\dots\squigto'\overline{t_m}\squigto'\overline{w_2}$. And since $\overline{w_1}\squigto'\overline{w_2}$, we also have that $\overline{z}\squigto'\overline{v_1}\squigto'\dots\squigto'\overline{v_{n-1}}\squigto'\overline{w_1}\squigto'\overline{w_2}$, giving us the reduction sequences we need.

        The reader may find it helpful to examine the following diagram:
        
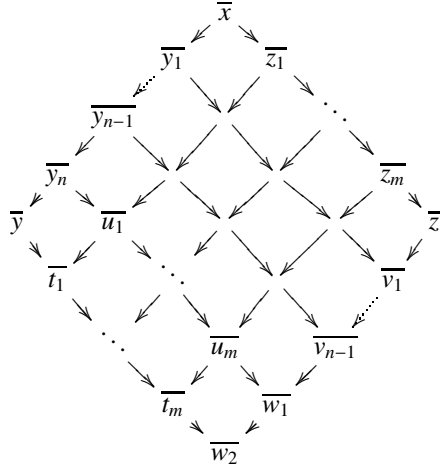
\begin{displaymath}
            \xymatrix@=1mm{
                & & & & \overline{x}\ar[dr]\ar[dl] & & & & \\
                & & & \overline{y_1}\ar[dr]\ar@{..>}[dl] & & \overline{z_1}\ar[dr]\ar[dl] & & & \\
                & & \overline{y_{n-1}}\ar[dr]\ar[dl] & & {}\ar[dr]\ar[dl] & & \ddots\ar[dr]\ar[dl] \\
                & \overline{y_n}\ar[dr]\ar[dl] & & {}\ar[dr]\ar[dl] & & {}\ar[dr]\ar[dl] & & \overline{z_m}\ar[dr]\ar[dl] & \\
                \overline{y}\ar[dr] & & \overline{u_1}\ar[dr]\ar[dl] & & {}\ar[dr]\ar[dl] & & {}\ar[dr]\ar[dl] & & \overline{z}\ar[dl]\\
                & \overline{t_1}\ar[dr] & & \ddots\ar[dr]\ar[dl] & & {}\ar[dr]\ar[dl] & & \overline{v_1}\ar@{..>}[dl] & \\
                & & \ddots\ar[dr] & & \overline{u_m}\ar[dr]\ar[dl] & & \overline{v_{n-1}}\ar[dl] & & \\
                & & & \overline{t_m}\ar[dr] & & \overline{w_1}\ar[dl] & & & \\
                & & & & \overline{w_2} & & & &
            }
        \end{displaymath}
    \end{proof}

    \begin{proof}[Proof of Theorem~\ref{th:canceling}]
        The `if' direction is obvious. For the `only if' direction, let $\overline{x}=x_1\dots x_n$, $\overline{y}=y_1\dots y_m$, and let $\overline{w}=w_1\dots w_k$. The proof is by induction on $n+m$.

        If $n+m=0$, then $\overline{x}=\overline{y}=\varepsilon$ and we are done. Now suppose we have the result for $n+m\leq J$ and let $n+m=J+1$.

        We split into cases. In principle there are 125 of them: one for each possible value for $w_1$, $x_n$ and $y_m$. But we can reduce that number a good bit.

        First up: if $w_1=l$ or $w_1=\rho$, then $\red(\overline{xw})=\overline{xw}$ and $\red(\overline{yw})=\overline{yw}$, so $\overline{x}=\overline{y}$ as needed. 

        Next up: note that if $x_n=y_m$, then $\overline{xw}$ and $\overline{yw}$ either both reduce or both don't reduce. In the former case, induction immediately finishes the job. In the latter, the same argument as in the previous paragraph does the work. Also the $(x_n=s,y_m=t)$ case and the $(x_n=t,y_m=s)$ cases are exactly parallel, so we need only consider one of each type.
        
        Finally, we can eliminate all remaining cases for which $w_1=n$, because we've already dealt with the one way for this to lead to a reduction, and thus all remaining cases don't reduce on either side and we've dealt with that sort of case already.

        We're left with the following 12 cases to consider: 
        \begin{center}\begin{tabular}{c|c|c}
            $w_1$ & $x_n$ & $y_m$ \\\hline
            $\lambda$ &  $l$   & $r$  \\
            $\lambda$ &  $l$   & $\lambda$  \\
            $\lambda$ &  $l$   & $\rho$  \\
            $\lambda$ &  $r$   & $\lambda$  \\
            $\lambda$ &  $r$ & $\rho$  \\
            $\lambda$ &  $\lambda$  & $\rho$  \\
            $r$ &  $l$   & $r$  \\
            $r$ &  $l$   & $\lambda$  \\
            $r$ &  $l$   & $\rho$  \\
            $r$ &  $r$   & $\lambda$  \\
            $r$ &  $r$   & $\rho$  \\
            $r$ &  $\lambda$   & $\rho$  \\
        \end{tabular}\end{center}
        It turns out that all of these cases are impossible.\footnote{Reminder: we're assuming that $\overline{x}$, $\overline{y}$, and $\overline{w}$ are all reduced and that $\overline{xw}$ and $\overline{yw}$ have the same reduction. It's only in that context that none of these cases are possible.} For example, the $\lambda$, $l$, $r$ case at the top cannot happen. Suppose for contradiction that it did. Then $\overline{xw}\squigto x_1\dots x_{n-1}\rho w_2\dots w_k$ while $\overline{yw}\squigto y_1\dots y_{m-1}w_2\dots w_k$. But then by the inductive hypothesis, $y_1\dots y_{m-1}=x_1\dots x_{n-1}\rho$, so $y_{m-1}=\rho$. But since $y_m=r$, it would then follow that $\overline{y}\not\in\rseq$, a contradiction. 

        The second case is even more directly impossible. In that case, $\red(\overline{xw})=x_1\dots x_{n-1}\rho w_2\dots w_k$, but $\red(\overline{yw})=y_1\dots y_{m-1}\lambda\lambda w_2\dots w_k$. But these simply aren't identical reductions, which is a contradiction.

        And now we leave the reader to check that each of the remaining 10 cases is also impossible in one of these two ways.
    \end{proof}

    \begin{proof}[Proof of Corollary~\ref{cor:replace}]
       Suppose $\red(\overline{z_1w})=\red(\overline{z_2w})$. Then by Corollary~\ref{cor:redred},\\$\red(\red(\overline{z_1})\red(\overline{w}))=\red(\red(\overline{z_2})\red(\overline{w}))$. So by Theorem~\ref{th:canceling}, $\red(\overline{z_1})=\red(\overline{z_2})$. Thus $\red(\red(\overline{z_1})\red(\overline{y}))=\red(\red(\overline{z_2})\red(\overline{y}))$. So again by Corollary~\ref{cor:redred}, $\red(\overline{z_1y})=\red(\overline{z_2y})$.
    \end{proof}

\bibliographystyle{plain}
\bibliography{biblio}

\end{document}